\documentclass[11pt,a4paper,reqno]{amsart}

\usepackage[a4paper,top=2.7cm,bottom=2.7cm,left=3cm,right=3cm,heightrounded]{geometry}
\usepackage[T1]{fontenc}
\usepackage[utf8]{inputenc}
\usepackage{lmodern}
\usepackage{microtype}
\usepackage{amsmath,amssymb,amsthm,mathtools}
\usepackage{mathrsfs}
\usepackage{graphicx}
\graphicspath{{./}{figures/}}
\usepackage{booktabs}
\usepackage{longtable}
\usepackage{array}
\usepackage{enumitem}
\usepackage{xurl}
\usepackage{tikz}
\usetikzlibrary{arrows.meta,calc,patterns,positioning}
\definecolor{chartRed}{RGB}{214,78,70}
\definecolor{chartGreen}{RGB}{54,150,88}
\definecolor{chartBlue}{RGB}{59,91,222}
\definecolor{chartGold}{RGB}{214,181,43}
\colorlet{chartRedFill}{chartRed!15}
\colorlet{chartGreenFill}{chartGreen!15}
\colorlet{chartBlueFill}{chartBlue!15}
\colorlet{chartGoldFill}{chartGold!18}

\usepackage[hidelinks]{hyperref}
\hypersetup{
  pdftitle={Infinitely many holes in connectedness loci for collinear affine iterated function systems},
  pdfauthor={Bernat Espigule},
  pdfsubject={Connectedness loci for collinear affine iterated function systems},
  pdfkeywords={restricted roots, connectedness locus, affine iterated function systems, holes, renormalization, Sturm theory}
}

\numberwithin{equation}{section}

\theoremstyle{plain}
\newtheorem{theorem}{Theorem}[section]
\newtheorem{proposition}[theorem]{Proposition}
\newtheorem{lemma}[theorem]{Lemma}
\newtheorem{corollary}[theorem]{Corollary}
\theoremstyle{definition}
\newtheorem{definition}[theorem]{Definition}
\theoremstyle{remark}
\newtheorem{remark}[theorem]{Remark}

\newcommand{\C}{\mathbb C}
\newcommand{\Rr}{\mathbb R}
\newcommand{\Z}{\mathbb Z}

\newcommand{\D}{\mathbb D}
\newcommand{\clD}{\overline{\mathbb D}}
\newcommand{\M}{\mathcal M}
\newcommand{\Rs}{\mathcal R}
\newcommand{\B}{\mathcal B}
\newcommand{\X}{\mathcal X}
\newcommand{\U}{\mathcal U}
\newcommand{\Ctrap}{\mathcal C_0}
\newcommand{\Enc}{\mathcal E_{\!\rm enc}}
\newcommand{\Hutch}{\mathcal H}
\newcommand{\Rop}{\mathscr R}
\newcommand{\Rstar}{\mathscr R_*}

\newcommand{\widW}{\widetilde W}
\newcommand{\ol}{\overline}
\newcommand{\eps}{\varepsilon}
\newcommand{\ii}{\mathrm i}
\newcommand{\restr}{\mathbin{\upharpoonright}}
\DeclareMathOperator{\Int}{int}
\DeclareMathOperator{\Impart}{Im}
\DeclareMathOperator{\Repart}{Re}

\title[Holes in connectedness loci of collinear affine IFS]{Infinitely many holes in connectedness loci for collinear affine iterated function systems}
\author[B. Espigule]{Bernat Espigule}
\address{Departament d'Inform\`atica, Matem\`atica Aplicada i Estad\'istica, Universitat de Girona, 17003 Girona, Catalonia, Spain}
\email{bernat@espigule.com}
\date{}

\begin{document}

\begin{abstract}
We study connectedness loci for a one-parameter family of collinear affine iterated function systems with equally spaced translations. These loci are equivalent to closures of roots of monic polynomials with coefficients in a prescribed finite interval of integers. We prove that each of them has infinitely many holes. The two-map case is the theorem of Calegari--Koch--Walker after inversion of the parameter. For all larger alphabets we construct a stationary family of finite-capture loops in the geometry of the corresponding difference attractor. Each loop surrounds a missing-center configuration whose midpoint word is not admissible, and a finite inverse-tree certificate places the enclosed parameter outside the connectedness locus. A no-return certificate separates the resulting holes. The witness parameters converge to a canonical algebraic boundary point. The finite sign, intersection, and pruning verifications are recorded as exact algebraic certificates.
\end{abstract}

\maketitle
\enlargethispage{2pt}

\section{Introduction}

Many fractal constructions depend on parameters, and the topology of the attractor may change in subtle ways as the parameter varies.  The classical prototype is the Mandelbrot set: it is the connectedness locus for the quadratic family $z\mapsto z^2+c$, and membership is detected by the orbit of the critical value.  For affine iterated function systems the analogous connectedness problem is usually less direct, since the maps need not be inverse branches of a single polynomial or rational map.

A fundamental affine example is the connectedness locus for a pair of complex linear maps introduced by Barnsley and Harrington~\cite{BarnsleyHarrington1985}.  In that setting the connectedness problem is closely related to zeros of power series with coefficients in $\{-1,0,1\}$ and to the geometry of complex Bernoulli convolutions; see, for instance, Solomyak and Xu~\cite{SolomyakXu2003}.  Bandt proved the existence of a genuine hole~\cite{Bandt2002}, and Calegari, Koch, and Walker later proved that there are infinitely many holes accumulating at a renormalization point~\cite{CalegariKochWalker2017}.  A key feature of this circle of ideas is that connectedness can be tested by a single marked point in a difference attractor, rather than by direct inspection of the original attractor.

The present paper studies the corresponding question for collinear systems with more than two maps.  The translations are equally spaced on the real line.  This has two consequences.  First, the parameter set can be described arithmetically as a closure of roots of monic polynomials whose non-leading coefficients lie in a finite interval of integers.  Second, the difference alphabet is again collinear and equally spaced: the difference of the $n$-map alphabet is the alphabet of the same type with parameter $N=2n-1$.  Thus the marked-point formulation remains in the same family of attractors, which makes it possible to combine symbolic finite capture with explicit restricted-root calculations.

Throughout the paper
\[
        \D:=\{z\in\C: |z|<1\},
\]
and all closures in $\C\setminus\clD$ are relative closures unless explicitly stated otherwise. Fix an integer $n\ge 2$ and set
\[
        D_n=\{-n+1,-n+2,\ldots,n-1\},\qquad
        A_n=\{-n+1,-n+3,\ldots,n-1\}.
\]
Let $\Rs_n$ be the set of roots of monic polynomials whose non-leading coefficients lie in $D_n$:
\begin{equation}\label{eq:Rn}
\Rs_n:=\left\{c\in\C:\; c^m+\sum_{j=0}^{m-1}d_jc^j=0\text{ for some }m\ge 1\text{ and }d_j\in D_n\right\}.
\end{equation}
For $c\in\C\setminus\clD$ consider the iterated function system
\[
        \left\{f_t(z)=t+\frac{z}{c}: t\in A_n\right\},
\]
and denote its attractor by $E(c,n)$:
\[
        E(c,n)=\bigcup_{t\in A_n}f_t(E(c,n)).
\]
The corresponding connectedness locus is
\[
        \M_n:=\{c\in\C\setminus\clD: E(c,n)\text{ is connected}\}.
\]

\begin{proposition}[Restricted-root formulation]\label{prop:restricted-root}
For every integer $n\ge 2$,
\begin{equation}\label{eq:Mn-root-formulation}
        \M_n=\overline{\Rs_n\cap(\C\setminus\clD)}^{\,\C}\cap(\C\setminus\clD)
        =\overline{\Rs_n\cap(\C\setminus\clD)}^{\,\C\setminus\clD}.
\end{equation}
In particular, $\M_n$ is closed in $\C\setminus\clD$. Equivalently,
\[
        c\in\M_n\quad\Longleftrightarrow\quad c\in\C\setminus\clD\text{ and }\exists(d_k)_{k\ge 1}\subset D_n\text{ such that }1+\sum_{k=1}^{\infty}d_kc^{-k}=0.
\]
\end{proposition}

\begin{proof}
The reciprocal-series criterion is the combination of the marked-point characterization with the radix description of the difference attractor; see Sections~\ref{app:restricted-roots} and~\ref{app:marked-point} for a derivation. If the digits $(d_k)$ are eventually zero, then clearing denominators gives a monic polynomial with non-leading coefficients in $D_n$, so the corresponding parameter lies in $\Rs_n$.

Now let $c\in\M_n$, and choose digits $(d_k)_{k\ge1}\subset D_n$ with
\[
        1+\sum_{k=1}^{\infty}d_kc^{-k}=0.
\]
Set
\[
        F(z):=1+\sum_{k=1}^{\infty}d_kz^{-k},\qquad
        F_m(z):=1+\sum_{k=1}^{m}d_kz^{-k}.
\]
Let $\eps>0$ be such that $B(c,\eps)\subset\C\setminus\clD$. Since $F(z)=1+O(z^{-1})$ at infinity, $F$ is not identically zero on $\C\setminus\clD$; hence its zeros are isolated. Choose $r_\eps\in(0,\eps)$ so that $F$ has no zero on
\[
        \Gamma_\eps:=\{z: |z-c|=r_\eps\}.
\]
The functions $F_m$ converge uniformly to $F$ on $\Gamma_\eps$, so by Rouch\'e's theorem, for all sufficiently large $m$ the function $F_m$ has the same number of zeros inside $\Gamma_\eps$ as $F$. In particular, for all such $m$ it has at least one zero there; choose one and call it $c_{\eps,m}$. Since
\[
        P_m(z):=z^mF_m(z)=z^m+\sum_{j=0}^{m-1}d_{m-j}z^j
\]
is monic with non-leading coefficients in $D_n$, one has $c_{\eps,m}\in\Rs_n$ and $|c_{\eps,m}-c|<\eps$. As $\eps>0$ is arbitrary, it follows that $c\in\overline{\Rs_n}^{\,\C}\cap(\C\setminus\clD)$.

Conversely, let $c_j\in\Rs_n\cap(\C\setminus\clD)$ with $c_j\to c\in\C\setminus\clD$. Write
\[
        1+\sum_{k=1}^{\infty}d_{j,k}c_j^{-k}=0
\]
with each $(d_{j,k})_{k\ge 1}$ eventually zero and $d_{j,k}\in D_n$. Because $D_n$ is finite, a diagonal subsequence yields digits $(d_k)_{k\ge1}\subset D_n$ such that, for each fixed $k$, one has $d_{j_\nu,k}=d_k$ for all sufficiently large $\nu$. Since $|c_{j_\nu}|\ge r>1$ eventually, the series are uniformly absolutely convergent, and dominated convergence gives
\[
        1+
        \sum_{k=1}^{\infty}d_kc^{-k}=0.
\]
Hence $c\in\M_n$. This proves \eqref{eq:Mn-root-formulation} and the closedness assertion.
\end{proof}

Proposition~\ref{prop:restricted-root} is the root-theoretic counterpart of the marked-point characterization recalled later in Proposition~\ref{prop:marked-point}: the same collinear family controls both the original attractor and its difference attractor, now with alphabet size $2n-1$. In the proof of Theorem~\ref{thm:main} we use Proposition~\ref{prop:restricted-root} only for this arithmetic framing and for the fact that $\M_n$ is closed in $\C\setminus\clD$; the geometric hole construction starts from the marked-point and trap/enclosure framework recalled in Section~\ref{sec:finite-capture}.

If one allows an arbitrary nonzero leading coefficient in \eqref{eq:Rn}, reciprocal polynomials show that the resulting root locus is invariant under $c\mapsto 1/c$. This is the symmetry behind the reciprocal parameterizations in the classical two-map literature. In particular, for $n=2$, after inversion $c\mapsto 1/c$, our locus $\M_2$ is exactly the connectedness locus studied by Calegari--Koch--Walker~\cite{CalegariKochWalker2017}. Hence the statement of the theorem below is already known in that case. In \cite{EspiguleJuherSaldana2024} it was conjectured that analogous holes exist in $\M_n$ for every $n>2$. The present paper proves that conjecture.

Unless explicitly stated otherwise, all boundaries, interiors, and connected components are taken relative to $\C\setminus\clD$. A hole of $\M_n$ is a bounded connected component of $(\C\setminus\clD)\setminus\M_n$.

\begin{theorem}[Infinitely many holes]\label{thm:main}
For every integer $n\ge 2$, the connectedness locus $\M_n$ has infinitely many holes.
\end{theorem}

For $n=2$, writing $\lambda:=1/c$, the affine map
\[
        T_\lambda(z)=\frac{1-\lambda}{2}z+\frac12
\]
conjugates our normalized IFS to the family studied by Calegari--Koch--Walker~\cite{CalegariKochWalker2017}. Thus the new contribution of the present paper concerns $n\ge 3$.

The structural core of the paper is Theorem~\ref{thm:stationary}. For each odd depth $k=1,3,5,\ldots$, it produces four admissible finite-capture charts whose selected boundary arcs form a Jordan curve $\Gamma_{n,k}\subset\M_n$.  The midpoint word of the four corresponding symbolic addresses is not admissible, so the central chart that would normally fill the loop is absent.  A finite inverse-tree certificate places an explicit witness parameter $c_{n,k}\notin\M_n$ in the bounded complementary component, and a no-return certificate separates the holes obtained at different odd depths.

The same stationary mechanism also identifies a canonical limit point. Corollary~\ref{cor:renorm-point} shows that the witness sequence converges to the unique zero, in a common witness box, of the explicit degree-$13$ polynomial $Q_n$ from Lemma~\ref{lem:exact-recursion}.  This gives an algebraic boundary point of $\M_n$ at which infinitely many of the holes accumulate.

The proof uses two finite-capture inputs from the companion preprint~\cite[Proposition 2.1 and Theorem 3.16]{EspiguleJuher2026}: the marked-point criterion and the canonical trap--enclosure framework. Section~\ref{sec:finite-capture} restates these inputs in the precise form needed below, and Appendix~\ref{app:imported-inputs} records the imported statements and the reductions used here. The restricted-root formulation is proved in the present paper from the marked-point criterion and the radix description of the difference attractor. The construction of the holes---the base symbolic configuration, the stationary transport, the witness-box analysis, the local branch geometry, and the inverse-tree extinction---is then carried out internally. The finite algebraic parts are organized in Appendix~\ref{app:exact-verification}; each item is reduced to an explicit sign, root-count, or non-intersection check. The parameter-space figures are finite-depth renderings obtained from the finite-capture algorithm summarized in Appendix~\ref{app:certified-rendering}; the proof does not rely on visual inspection of the figures.

\begin{figure}[t]
\centering
\includegraphics[width=0.98\textwidth]{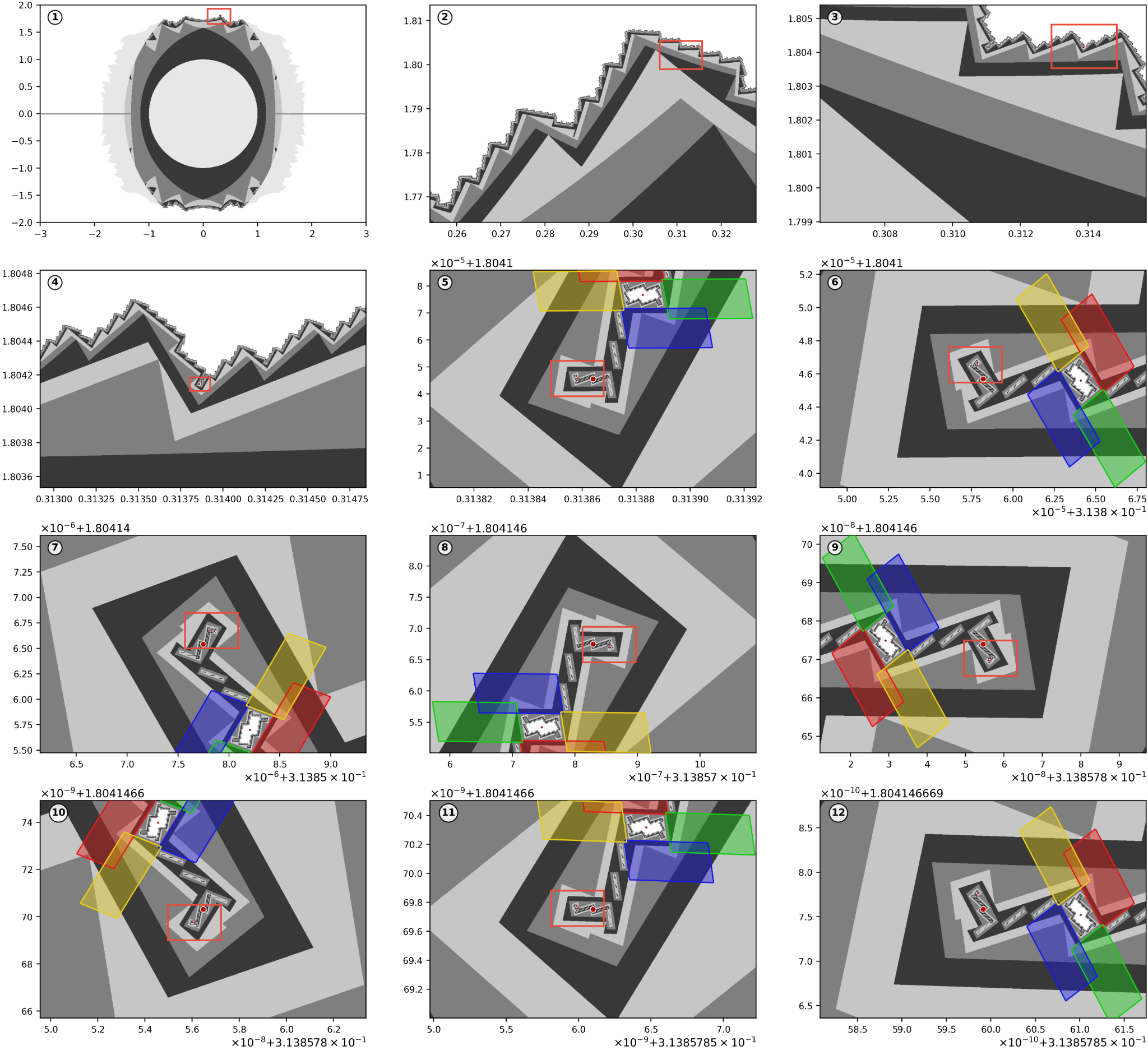}
\caption{Finite-depth zooms of $\M_3$ along the stationary family, generated by the finite-capture inverse-iteration procedure. In panels 1--11 the highlighted rectangle marks the region magnified in the next panel. The first visible hole appears in panel 5. In panels 5--12 the marked point is the witness parameter at the displayed level. The rendering protocol is described in Appendix~\ref{app:certified-rendering}.}
\label{fig:zooms}
\end{figure}

Figure~\ref{fig:zooms} provides a finite-depth guide for $n=3$. Each panel magnifies the highlighted rectangle from the preceding one; the first visible hole appears in panel 5. Figures~\ref{fig:center-combinatorics} and~\ref{fig:base-four-chart} summarize the missing-center mechanism and the base four-chart configuration, while Figure~\ref{fig:n345-comparison} gives side-by-side finite-depth renderings for $n=3,4,5$. The paper is organized as follows. Section~\ref{sec:finite-capture} records the finite-capture inputs and the loop-witness notation. Section~\ref{sec:ghost} constructs the base symbolic configuration and the corresponding loop. Section~\ref{sec:transport} introduces the stationary renormalization and proves the infinite family theorem. Section~\ref{sec:proof-main} deduces Theorem~\ref{thm:main} and identifies the algebraic accumulation point. Appendix~\ref{app:imported-inputs} records the imported inputs in the form used here; Appendix~\ref{app:stationary-formulas} collects the explicit stationary formulas; Appendix~\ref{app:exact-verification} contains the finite algebraic verifications used in the proof, and Appendix~\ref{app:certified-rendering} records the rendering protocol for the parameter-space figures.

\section{Finite-capture toolkit and loop witnesses}\label{sec:finite-capture}

We use the finite-capture framework developed in~\cite{EspiguleJuher2026} only through the marked-point criterion and the canonical trap--enclosure statement restated in Propositions~\ref{prop:marked-point} and~\ref{prop:trap-enclosure}. The restricted-root formulation, Proposition~\ref{prop:restricted-root}, is then derived in the present paper from these inputs. Appendix~\ref{app:imported-inputs} records the exact imported statements and the short reductions used here.

Set
\[
        N:=2n-1.
\]
The open two-disk lens is
\[
        \X_n:=\{c\in\C\setminus\clD: |c\pm1|<\sqrt{2n}\}.
\]
All parameters constructed in this paper lie in $\X_n\setminus\Rr$.

\subsection*{Difference attractor and backward polynomials}

The difference digit set is
\[
        A_N=A_n-A_n=2D_n=\{-(N-1),-(N-3),\ldots,N-3,N-1\}\subset 2\Z.
\]
Let $E(c,N)$ denote the attractor of the associated difference IFS $f_t(z)=t+z/c$, $t\in A_N$. For $S\subset\C$, define the Hutchinson operator of the difference IFS by
\[
        \Hutch_{c,N}(S):=\bigcup_{t\in A_N}f_t(S).
\]
The first imported input is the marked-point criterion, proved in~\cite[Proposition 2.1]{EspiguleJuher2026} in the present normalization.

\begin{proposition}[Marked-point characterization]\label{prop:marked-point}
For every $c\in\C\setminus\clD$ one has
\[
        c\in\M_n\quad\Longleftrightarrow\quad 2c\in E(c,N).
\]
\end{proposition}

\begin{proof}
This is~\cite[Proposition 2.1]{EspiguleJuher2026}. Section~\ref{app:marked-point} recalls how the criterion enters the present normalization.
\end{proof}

\begin{proposition}[Canonical trap and enclosure]\label{prop:trap-enclosure}
For every $c\in\X_n\setminus\Rr$, the companion paper~\cite{EspiguleJuher2026} provides a canonical open trap $\Ctrap(c)$ and a canonical closed enclosure $\Enc(c,N)$ with
\[
        \Ctrap(c)\subset \Int(E(c,N)),\qquad E(c,N)\subset \Enc(c,N),
        \qquad
        \Ctrap(c)\subset\Hutch_{c,N}(\Ctrap(c)).
\]
The trap $\Ctrap(c)$ is an open parallelogram centered at $0$. The explicit half-width formulas used later are recalled in Appendices~\ref{app:stationary-formulas} and~\ref{app:exact-verification}.
\end{proposition}

\begin{proof}
This is the restriction to $\X_n\setminus\Rr$ of the canonical trap--enclosure framework of~\cite[Definition 3.11, Corollary 3.13, Proposition 3.15, and Theorem 3.16]{EspiguleJuher2026}.
\end{proof}

\paragraph{Dependency convention.}
Every later appeal to~\cite{EspiguleJuher2026} factors through Propositions~\ref{prop:marked-point} and~\ref{prop:trap-enclosure}; Appendix~\ref{app:imported-inputs} records these inputs and the related reductions used here.

For $t\in A_N$ define the inverse branch $g_t(z)=c(z-t)$. For a finite word $u=u_1\cdots u_m\in\Z^m$, write
\[
        g_u:=g_{u_m}\circ\cdots\circ g_{u_1}.
\]
If $w=(w_1,\ldots,w_m)\in\Z^m$ is a finite integer word, set
\[
        Z_w(c):=g_w(2c)=2c^{m+1}-\sum_{j=1}^m w_jc^{m+1-j}
\]
and
\[
        W_w(c):=\frac{Z_w(c)}{c}=2c^m-\sum_{j=1}^m w_jc^{m-j}.
\]
Since all parameters under consideration lie in $\C\setminus\clD$, the equations $Z_w(c)=0$ and $W_w(c)=0$ are equivalent.

\subsection*{Finite-capture charts and admissible center roots}

For each $c\in\X_n\setminus\Rr$, let $\Ctrap(c)$ denote the canonical trap from Proposition~\ref{prop:trap-enclosure}. Thus
\[
        \Ctrap(c)\subset \Int(E(c,N)),\qquad
        \Ctrap(c)\subset\Hutch_{c,N}(\Ctrap(c)).
\]
This gives open parameter charts
\begin{equation}\label{eq:chart-def}
        \U_u:=\{c\in\X_n\setminus\Rr: g_u(2c)=Z_u(c)\in\Ctrap(c)\},\qquad u\in A_N^m.
\end{equation}

\begin{lemma}[Chart inclusion]\label{lem:chart-inclusion}
For every $u\in A_N^m$, the chart $\U_u$ is open and satisfies $\U_u\subset\Int(\M_n)$. Consequently,
\[
        \partial_{\X_n\setminus\Rr}\U_u\subset\M_n.
\]
\end{lemma}

\begin{proof}
If $c\in\U_u$, then $Z_u(c)=g_u(2c)\in\Ctrap(c)\subset\Int(E(c,N))$ by Proposition~\ref{prop:trap-enclosure}. Applying $f_u$ gives $2c\in E(c,N)$, so $c\in\M_n$ by Proposition~\ref{prop:marked-point}. Because membership in $\Ctrap(\cdot)$ is determined by strict inequalities in continuous functions of the parameter, the same witness persists on a neighborhood of $c$. Hence $c\in\Int(\M_n)$ and $\U_u$ is open. Since $\M_n$ is closed in $\C\setminus\clD$ by Proposition~\ref{prop:restricted-root}, its relative boundary in $\X_n\setminus\Rr$ satisfies $\partial_{\X_n\setminus\Rr}\U_u\subset\M_n$.
\end{proof}

\begin{proposition}[Admissible center-root chart criterion]\label{prop:center-root}
Let $u\in A_N^m$ and let $c\in\X_n\setminus\Rr$ satisfy $W_u(c)=0$. Then $c\in\U_u\subset\Int(\M_n)$.
\end{proposition}

\begin{proof}
By Proposition~\ref{prop:trap-enclosure}, the canonical trap $\Ctrap(c)$ is an open parallelogram centered at $0$ and contained in $\Int(E(c,N))$. Thus $0\in\Ctrap(c)$. If $W_u(c)=0$, then $Z_u(c)=0$, so $c\in\U_u$ by \eqref{eq:chart-def}. Now apply Lemma~\ref{lem:chart-inclusion}.
\end{proof}

This is the center-root criterion used below: if $u\in A_N^m$ and $W_u(c)=0$, then $c\in\U_u$. We call the connected component of $\U_u$ containing $c$ the center region determined by the admissible center root $(u,c)$.

\begin{lemma}[Quadrant semialgebraicity of charts]\label{lem:semialgebraic}
For every finite word $u\in A_N^m$, the restriction of $\U_u$ to any open quadrant is semialgebraic. Consequently, on each open quadrant its boundary is contained in a finite union of real-algebraic arcs and points.
\end{lemma}

\begin{proof}
On the upper half-plane, Appendix~\ref{app:stationary-formulas} rewrites the condition $c\in\U_u$ as the four polynomial sign inequalities \eqref{eq:chart-signs}. On any other open quadrant the same argument applies after fixing the signs of $x=\Repart c$ and $y=\Impart c$. Hence $\U_u$ is semialgebraic on each open quadrant. Standard semialgebraic stratification gives the boundary statement.
\end{proof}

Away from finitely many singularities these boundary components are real-analytic arcs, which is the regularity used when boundary pieces are concatenated into Jordan curves.

\subsection*{Exteriority and loop witnesses}

By Proposition~\ref{prop:trap-enclosure}, the same framework also provides a canonical closed enclosure $\Enc(c,N)\supset E(c,N)$. This turns exteriority into a finite inverse-tree question. For $m\ge0$ and $c\in\X_n\setminus\Rr$, let
\[
        T_m(c):=\{u=t_1\cdots t_m\in A_N^m: g_{t_1\cdots t_j}(2c)\in\Enc(c,N)\text{ for every }1\le j\le m\}.
\]
Here $T_0(c):=\{\varnothing\}$. If $T_m(c)=\varnothing$ for some $m$, then no admissible inverse orbit of $2c$ can remain inside the enclosure, hence $c\notin\M_n$.

\begin{proposition}[Finite inverse-tree criterion]\label{prop:inverse-tree}
Let $c\in\X_n\setminus\Rr$. If $T_m(c)=\varnothing$ for some $m\ge0$, then $c\notin\M_n$.
\end{proposition}

\begin{proof}
Assume for contradiction that $c\in\M_n$. By Proposition~\ref{prop:marked-point}, one has $2c\in E(c,N)$. Choose an address $(t_j)_{j\ge1}\subset A_N$ for $2c$ in the difference attractor. Then every partial inverse iterate
\[
        g_{t_1\cdots t_j}(2c)\in E(c,N)\subset\Enc(c,N)\qquad (j\ge1),
\]
so every prefix $t_1\cdots t_m$ belongs to $T_m(c)$. This contradicts $T_m(c)=\varnothing$. Hence $c\notin\M_n$.
\end{proof}

\begin{definition}[Loop witness]\label{def:loop-witness}
A loop witness for $\M_n$ is a pair $(\Gamma,c_*)$ such that
\begin{enumerate}[label=(\roman*)]
\item $\Gamma\subset\M_n$ is a Jordan curve;
\item $c_*$ lies in the bounded component of $\C\setminus\Gamma$;
\item $c_*\in(\C\setminus\clD)\setminus\M_n$.
\end{enumerate}
\end{definition}

\begin{lemma}[A loop witness forces a hole]\label{lem:loop-hole}
If $(\Gamma,c_*)$ is a loop witness, then $\M_n$ has a hole contained in the bounded component of $\C\setminus\Gamma$.
\end{lemma}

\begin{proof}
Let $H$ be the connected component of $(\C\setminus\clD)\setminus\M_n$ containing $c_*$. Since $H$ is connected, disjoint from $\Gamma$, and contains a point of the bounded component of $\C\setminus\Gamma$, the Jordan curve theorem gives $H$ inside that bounded component. Thus $H$ is bounded.
\end{proof}

\begin{proposition}[Four-chart loop criterion]\label{prop:four-chart}
Let $C_1,C_2,C_3,C_4\subset\X_n\setminus\Rr$ be open sets such that
\[
        \partial_{\X_n\setminus\Rr}C_i\subset\M_n\qquad(i=1,2,3,4).
\]
Assume that there exist pairwise distinct points
\[
        p_{12},p_{23},p_{34},p_{41}\in\X_n\setminus\Rr
\]
and real-analytic arcs
\[
        \gamma_1\subset\partial_{\X_n\setminus\Rr}C_1,
        \quad \gamma_2\subset\partial_{\X_n\setminus\Rr}C_2,
        \quad \gamma_3\subset\partial_{\X_n\setminus\Rr}C_3,
        \quad \gamma_4\subset\partial_{\X_n\setminus\Rr}C_4
\]
with the cyclic incidence pattern
\[
        \gamma_1:p_{12}\to p_{41},\quad
        \gamma_2:p_{23}\to p_{12},\quad
        \gamma_3:p_{34}\to p_{23},\quad
        \gamma_4:p_{41}\to p_{34},
\]
such that the four arcs meet only at their endpoints and meet there transversely. Then
\[
        \Gamma:=\gamma_1\cup\gamma_2\cup\gamma_3\cup\gamma_4
\]
is a Jordan curve contained in $\M_n$.
\end{proposition}

\begin{proof}
The transversality and disjointness assumptions imply that the union is a simple closed curve. Since each arc lies in one of the relative boundaries $\partial_{\X_n\setminus\Rr}C_i\subset\M_n$, we get $\Gamma\subset\M_n$.
\end{proof}

\section{Ghost vacancies and the base loop}\label{sec:ghost}

Inside the nonreal lens, an admissible word $u\in A_N^m$ with $W_u(c)=0$ places $c$ in the center region of the chart $\U_u$ by Proposition~\ref{prop:center-root}. The hole mechanism comes from the complementary situation: the exact-capture equation persists, but the center word leaves the admissible lattice and the corresponding center chart is missing.

\begin{definition}[Ghost parallelogram]\label{def:ghost}
Define the lattice completion
\[
        \widehat A_N:=\{-(N-1),-(N-2),\ldots,N-1\}\subset\Z.
\]
A ghost parallelogram of depth $m$ is a quintuple
\[
        (d^{(1)},d^{(2)},d^{(3)},d^{(4)};w)
\]
with $d^{(i)}\in A_N^m$ and $w\in\widehat A_N^m\setminus A_N^m$ such that
\[
        d^{(1)}+d^{(3)}=d^{(2)}+d^{(4)}=2w.
\]
The word $w$ is the ghost center word. Since $A_N\subset2\Z$, the condition $w\in\widehat A_N^m\setminus A_N^m$ is equivalent to saying that $w$ has at least one odd digit.
\end{definition}

\begin{proposition}[Ghost-vacancy principle]\label{prop:ghost-vacancy}
Let $(d^{(1)},d^{(2)},d^{(3)},d^{(4)};w)$ be a ghost parallelogram of depth $m$, and let $c_*\in\X_n\setminus\Rr$ satisfy $W_w(c_*)=0$. Assume that boundary arcs of the four charts
\[
        \U_{d^{(1)}},\quad \U_{d^{(2)}},\quad \U_{d^{(3)}},\quad \U_{d^{(4)}}
\]
form a Jordan curve $\Gamma$ around $c_*$, and that $c_*\notin\M_n$. Then $(\Gamma,c_*)$ is a loop witness. If, by contrast, the center word were admissible, then the same equation $W_w(c_*)=0$ would place $c_*$ in the center region of the chart $\U_w$ by Proposition~\ref{prop:center-root}; see Figure~\ref{fig:center-combinatorics}.
\end{proposition}

\begin{proof}
The first assertion is immediate from Definition~\ref{def:loop-witness}, and the second from Proposition~\ref{prop:center-root}.
\end{proof}

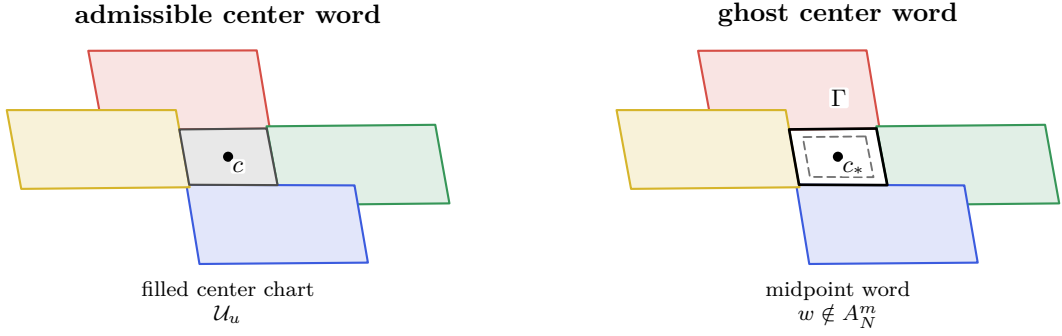
\begin{figure}[t]
\centering
\begin{tikzpicture}[x=0.72cm,y=0.72cm, line join=round, line cap=round]
  \tikzset{
    paneltitle/.style={font=\small\bfseries},
    callout/.style={font=\scriptsize, fill=white, inner sep=1.2pt, rounded corners=1pt},
    pointlabel/.style={font=\footnotesize, fill=white, inner sep=0.8pt, rounded corners=1pt},
    looplabel/.style={font=\footnotesize, fill=white, inner sep=0.8pt, rounded corners=1pt}
  }
  \def\OuterCharts{%
    \path[fill=chartRedFill, draw=chartRed, line width=0.8pt]
      (-2.297,0.489) -- (-2.568,1.944) -- (0.526,1.951) -- (0.766,0.515) -- cycle;
    \path[fill=chartGreenFill, draw=chartGreen, line width=0.8pt]
      (0.976,-0.887) -- (0.706,0.558) -- (3.805,0.586) -- (4.065,-0.856) -- cycle;
    \path[fill=chartBlueFill, draw=chartBlue, line width=0.8pt]
      (-0.526,-1.968) -- (-0.766,-0.514) -- (2.322,-0.526) -- (2.568,-1.943) -- cycle;
    \path[fill=chartGoldFill, draw=chartGold, line width=0.8pt]
      (-3.799,-0.590) -- (-4.065,0.856) -- (-0.959,0.856) -- (-0.706,-0.560) -- cycle;
  }
  \def\CenterQuad{(0.707,0.518) -- (0.908,-0.521) -- (-0.714,-0.514) -- (-0.896,0.501) -- cycle}
  \def\InnerGhost{($(0,0)!0.72!(0.707,0.518)$) --
      ($(0,0)!0.72!(0.908,-0.521)$) --
      ($(0,0)!0.72!(-0.714,-0.514)$) --
      ($(0,0)!0.72!(-0.896,0.501)$) -- cycle}

  \begin{scope}[shift={(-5.6,0)}]
    \node[paneltitle] at (0,2.62) {admissible center word};
    \OuterCharts
    \path[fill=gray!18, draw=black!70, line width=0.8pt] \CenterQuad;
    \fill (0,0) circle (1.9pt);
    \node[pointlabel, below right=1pt] at (0,0) {$c$};
    \node[callout, align=center, anchor=north] at (0,-2.22) {filled center chart\\[-1pt] $\U_u$};
  \end{scope}

  \begin{scope}[shift={(5.6,0)}]
    \node[paneltitle] at (0,2.62) {ghost center word};
    \OuterCharts
    \draw[black, line width=1.0pt] \CenterQuad;
    \path[draw=black!55, densely dashed, line width=0.7pt] \InnerGhost;
    \fill (0,0) circle (1.9pt);
    \node[pointlabel, below right=1pt] at (0,0) {$c_*$};
    \node[looplabel, anchor=south] at (0.02,0.88) {$\Gamma$};
    \node[callout, align=center, anchor=north] at (0,-2.22) {midpoint word\\[-1pt] $w\notin A_N^m$};
  \end{scope}
\end{tikzpicture}
\caption{Local center-chart combinatorics, shown schematically and not to scale. Left: if the center word $u\in A_N^m$ is admissible and $W_u(c)=0$, then the center-root criterion places $c$ in the distinguished connected component of the center chart $\U_u$. Right: if the midpoint word $w\notin A_N^m$, that center chart is absent. The dashed inner quadrilateral marks the missing center-chart position, while the thick black quadrilateral indicates the surrounding Jordan loop $\Gamma\subset\M_n$ enclosing a ghost parameter $c_*\notin\M_n$.}
\label{fig:center-combinatorics}
\end{figure}

From this point on we fix an integer $n\ge3$ and write
\[
        N:=2n-1,
        \qquad a:=N-1=2(n-1),
        \qquad b:=N-3=2(n-2).
\]
Thus $0,\pm a,\pm b,\pm2\in A_N$.

\subsection*{Base words and the base loop}
Define four words of length $18$ in the alphabet $A_N$ by
\begin{align}
 d_1^{(1)}&=(0,-a,-a,a,a,-a,-a,0,a,b,-a,-a,2,a,a,-a,-a,-a),\notag\\
 d_1^{(2)}&=(0,-a,-a,a,a,-a,-a,0,a,a,-a,0,a,0,-a,-a,a,0),\label{eq:base-words}\\
 d_1^{(3)}&=(0,-a,-a,a,a,-a,-a,0,a,a,-a,0,a,0,-a,-b,a,a),\notag\\
 d_1^{(4)}&=(0,-a,-a,a,a,-a,-a,0,a,b,-a,-a,2,a,a,-b,-a,0).\notag
\end{align}
Let $\U_1^{(i)}:=\U_{d_1^{(i)}}$ be the corresponding level-$18$ charts.

\begin{proposition}[Base four-chart loop]\label{prop:base-loop}
For every $n\ge3$ there exist real-analytic arcs
\[
        \gamma_i\subset\partial_{\X_n\setminus\Rr}\U_1^{(i)}\qquad(i=1,2,3,4)
\]
meeting cyclically and transversely at four corner points such that
\[
        \Gamma_{n,1}:=\gamma_1\cup\gamma_2\cup\gamma_3\cup\gamma_4
\]
is a Jordan curve contained in $\M_n$.
\end{proposition}

\begin{proof}
By Lemma~\ref{lem:chart-inclusion}, each relative boundary $\partial_{\X_n\setminus\Rr}\U_1^{(i)}$ is contained in $\M_n$. By Proposition~\ref{prop:branch-selection} and Lemmas~\ref{lem:corner-systems} and~\ref{lem:local-sign-region}, the relevant local boundary branches of the four charts are real-analytic, meet transversely at four consecutive corner points, have no nonconsecutive intersections, and form the boundary of the certified local sign region. Hence the hypotheses of Proposition~\ref{prop:four-chart} are satisfied, and the four arcs form a Jordan curve $\Gamma_{n,1}\subset\M_n$. See Figure~\ref{fig:base-four-chart} for the geometry when $n=3$.
\end{proof}

\begin{figure}[t]
\centering
\begin{minipage}[t]{0.47\linewidth}
  \centering
  \includegraphics[width=\linewidth]{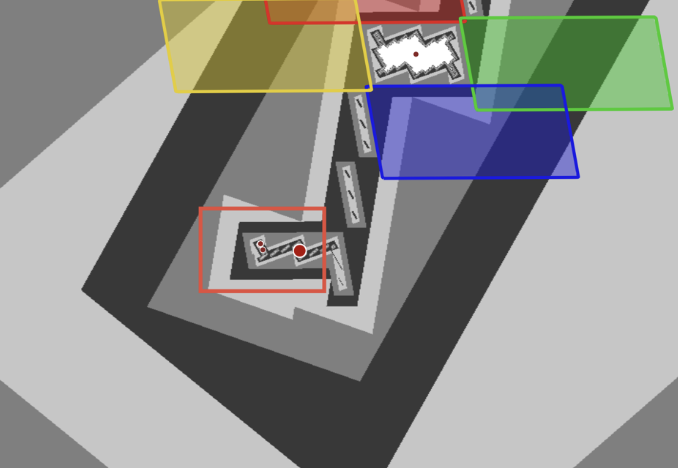}
\end{minipage}\hfill
\begin{minipage}[t]{0.47\linewidth}
\centering
\resizebox{\linewidth}{!}{%
\begin{tikzpicture}[x=0.78cm,y=0.78cm, line join=round, line cap=round]
  \tikzset{
    looppt/.style={font=\footnotesize, fill=white, inner sep=0.6pt, rounded corners=1pt},
    loopname/.style={font=\footnotesize, fill=white, inner sep=0.8pt, rounded corners=1pt}
  }
  \path[fill=chartRedFill, draw=chartRed, line width=0.8pt]
    (-2.297,0.489) -- (-2.568,1.944) -- (0.526,1.951) -- (0.766,0.515) -- cycle;
  \path[fill=chartGreenFill, draw=chartGreen, line width=0.8pt]
    (0.976,-0.887) -- (0.706,0.558) -- (3.805,0.586) -- (4.065,-0.856) -- cycle;
  \path[fill=chartBlueFill, draw=chartBlue, line width=0.8pt]
    (-0.526,-1.968) -- (-0.766,-0.514) -- (2.322,-0.526) -- (2.568,-1.943) -- cycle;
  \path[fill=chartGoldFill, draw=chartGold, line width=0.8pt]
    (-3.799,-0.590) -- (-4.065,0.856) -- (-0.959,0.856) -- (-0.706,-0.560) -- cycle;
  \node[chartRed] at (-0.96,1.46) {$\U^{(1)}_1$};
  \node[chartGreen] at (2.47,0.03) {$\U^{(2)}_1$};
  \node[chartBlue] at (0.95,-1.48) {$\U^{(3)}_1$};
  \node[chartGold] at (-2.47,-0.02) {$\U^{(4)}_1$};
  \coordinate (p12) at (0.707,0.518);
  \coordinate (p23) at (0.908,-0.521);
  \coordinate (p34) at (-0.714,-0.514);
  \coordinate (p41) at (-0.896,0.501);
  \draw[very thick] (p12) -- (p23) -- (p34) -- (p41) -- cycle;
  \node[loopname] at (-0.02,0.86) {$\Gamma_{3,1}$};
  \fill (0,0) circle (2.0pt);
  \node[looppt, below right=1pt] at (0,0) {$c_{3,1}$};
  \fill (p12) circle (1.3pt);
  \node[looppt, above right=1pt] at (p12) {$p_{12}$};
  \fill (p23) circle (1.3pt);
  \node[looppt, below right=1pt] at (p23) {$p_{23}$};
  \fill (p34) circle (1.3pt);
  \node[looppt, below left=1pt] at (p34) {$p_{34}$};
  \fill (p41) circle (1.3pt);
  \node[looppt, above left=1pt] at (p41) {$p_{41}$};
\end{tikzpicture}%
}
\end{minipage}
\caption{Base four-chart configuration for $n=3$. Left: numerical rendering of the first stationary hole and its witness $c_{3,1}$. Right: schematic combinatorics of the four-chart configuration. The points $p_{12},p_{23},p_{34},p_{41}$ are the consecutive intersections of the relevant boundary arcs, and $\Gamma_{3,1}$ denotes the resulting Jordan curve.}
\label{fig:base-four-chart}
\end{figure}
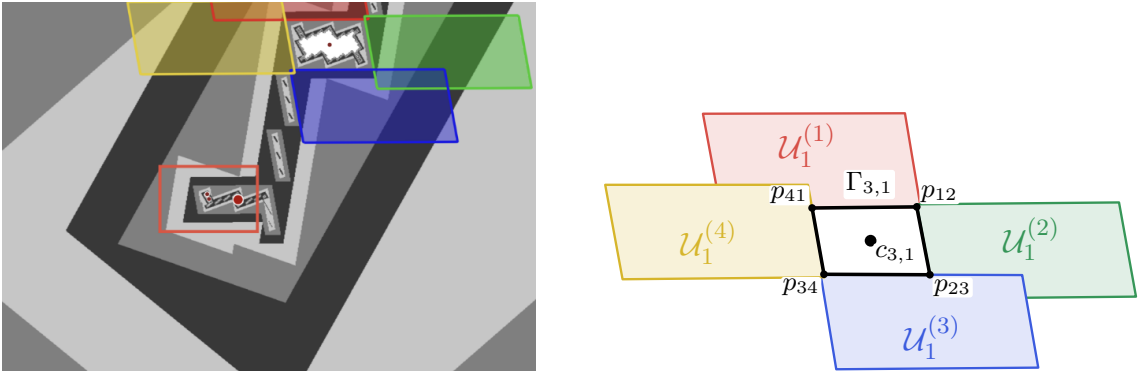

\subsection*{The base ghost word and its witness}
For later use, define the common witness box
\begin{equation}\label{eq:witness-box}
        \B_n:=\left\{c=x+\ii y\in\C:\frac14<x<\frac12,\quad \sqrt n<y<\sqrt n+\frac18\right\}.
\end{equation}
Define the ghost center word
\begin{equation}\label{eq:w1}
        w_1:=\frac{d_1^{(1)}+d_1^{(3)}}2=\frac{d_1^{(2)}+d_1^{(4)}}2.
\end{equation}
Its tenth digit is $(a+b)/2=2n-3$, which is odd; since $A_N\subset2\Z$, this shows that $w_1\notin A_N^{18}$. Consequently the base words form a ghost parallelogram.

\begin{proposition}[Base ghost vacancy]\label{prop:base-ghost}
For every $n\ge3$, the quintuple
\[
        (d_1^{(1)},d_1^{(2)},d_1^{(3)},d_1^{(4)};w_1)
\]
is a ghost parallelogram of depth $18$, and the pair $(\Gamma_{n,1},c_{n,1})$ is a loop witness for $\M_n$, where $c_{n,1}$ is the unique zero of $W_{w_1}$ in the witness box $\B_n$.
\end{proposition}

\begin{proof}
The midpoint identity \eqref{eq:w1} and the odd tenth digit show that the five words form a ghost parallelogram. By Proposition~\ref{prop:base-loop}, $\Gamma_{n,1}$ is a Jordan curve contained in $\M_n$. By Lemma~\ref{lem:local-sign-region}, the bounded component of $\C\setminus\Gamma_{n,1}$ is the sign region cut out by the explicit branch functions \eqref{eq:F1}--\eqref{eq:F4}. The witness-box analysis in Lemmas~\ref{lem:witness-sign} and~\ref{lem:uniform-isolation} places the isolated root $c_{n,1}$ in that bounded component. By Lemma~\ref{lem:extinction}, one has $T_{17}(c_{n,1})=\varnothing$, so Proposition~\ref{prop:inverse-tree} gives $c_{n,1}\notin\M_n$. Hence Proposition~\ref{prop:ghost-vacancy} yields the loop witness $(\Gamma_{n,1},c_{n,1})$.
\end{proof}

\section{Stationary transport of the vacancy and infinitely many holes}\label{sec:transport}

The base ghost vacancy is propagated by a stationary operator on words.

\subsection*{The renormalization operator}
Let
\[
        P:=(0,-a,-a,a,a,-a,-a,0,a,b,-a,-a)\in A_N^{12}.
\]
For an integer word $u=(u_1,\ldots,u_m)\in\Z^m$ with $m\ge10$ define
\begin{equation}\label{eq:Rop}
        \Rop(u):=(P,a-u_{10},-u_{11},-u_{12},\ldots,-u_m)\in\Z^{m+3}.
\end{equation}
If $d=(d^{(1)},d^{(2)},d^{(3)},d^{(4)})$ is an ordered $4$-tuple of words, write
\[
        \Rop(d):=(\Rop(d^{(1)}),\Rop(d^{(2)}),\Rop(d^{(3)}),\Rop(d^{(4)})).
\]

\begin{lemma}[Basic properties of $\Rop$]\label{lem:R-basic}
Let $m\ge10$.
\begin{enumerate}[label=(\roman*)]
\item $|\Rop(u)|=|u|+3$ for every $u\in\Z^m$.
\item If $u\in A_N^m$ and $u_{10}\ge0$, then $\Rop(u)\in A_N^{m+3}$.
\item $\Rop$ preserves midpoints:
\[
        \Rop\left(\frac{u+v}{2}\right)=\frac{\Rop(u)+\Rop(v)}2.
\]
\end{enumerate}
\end{lemma}

\begin{proof}
Part (i) is immediate. For (ii), the prefix $P$ lies in $A_N^{12}$. If $u\in A_N^m$ and $u_{10}\ge0$, then $0\le u_{10}\le a$ and $u_{10}$ is even, so $a-u_{10}\in A_N$. The remaining tail digits are sign changes of digits already in $A_N$. Part (iii) follows directly from \eqref{eq:Rop}.
\end{proof}

\subsection*{Transported words and transported defects}
Define recursively
\[
        d_k:=\Rop^{k-1}(d_1),\qquad w_k:=\Rop^{k-1}(w_1)\qquad(k\ge1),
\]
where $d_1=(d_1^{(1)},d_1^{(2)},d_1^{(3)},d_1^{(4)})$. Write
\[
        d_k=(d_k^{(1)},d_k^{(2)},d_k^{(3)},d_k^{(4)}).
\]
By Lemma~\ref{lem:R-basic}(i), all five words have common length
\[
        m_k:=|d_k^{(i)}|=18+3(k-1)=3k+15.
\]
By Lemma~\ref{lem:R-basic}(iii), the midpoint relation persists:
\begin{equation}\label{eq:midpoint-k}
        w_k=\frac{d_k^{(1)}+d_k^{(3)}}2=\frac{d_k^{(2)}+d_k^{(4)}}2.
\end{equation}

\begin{lemma}[Transport of the lattice defect]\label{lem:lattice-defect}
For every $k\ge1$, the word $w_k$ does not belong to $A_N^{m_k}$. More precisely, its first odd digit occurs at index
\[
        j_k:=10+3(k-1)=m_k-8.
\]
\end{lemma}

\begin{proof}
For $k=1$ the tenth digit of $w_1$ is $2n-3$, which is odd. The rule \eqref{eq:Rop} prepends the even block $P$ and then shifts the remaining digits three places to the right, changing only signs. Thus the parity of each existing digit is preserved and the first odd digit index increases by $3$ at each step.
\end{proof}

\begin{corollary}[Transported ghost parallelograms]\label{cor:transported-ghosts}
For every $k\ge1$, the quintuple
\[
        (d_k^{(1)},d_k^{(2)},d_k^{(3)},d_k^{(4)};w_k)
\]
is a ghost parallelogram of depth $m_k$.
\end{corollary}

\begin{proof}
The midpoint identity \eqref{eq:midpoint-k} gives
\[
        d_k^{(1)}+d_k^{(3)}=d_k^{(2)}+d_k^{(4)}=2w_k.
\]
For $k=1$ the four corner words belong to $A_N^{18}$ by construction, and their tenth digits are $b,a,a,b$, hence nonnegative. If $d_k^{(i)}\in A_N^{m_k}$ has nonnegative tenth digit, then Lemma~\ref{lem:R-basic}(ii) gives $d_{k+1}^{(i)}=\Rop(d_k^{(i)})\in A_N^{m_{k+1}}$. Moreover, every transported word has tenth digit equal to the fixed prefix digit $b\ge0$. An induction on $k$ therefore shows that all four corner words lie in $A_N^{m_k}$ for every $k\ge1$. By Lemma~\ref{lem:lattice-defect}, the center word $w_k$ does not belong to $A_N^{m_k}$.
\end{proof}

For each $k\ge1$ define the level-$m_k$ charts
\[
        \U_k^{(i)}:=\U_{d_k^{(i)}}\qquad(i=1,2,3,4).
\]
At level $k$ the witness polynomial is
\[
        W_{w_k}(c)=2c^{m_k}-\sum_{j=1}^{m_k}(w_k)_jc^{m_k-j}.
\]
The root-isolation step is organized around a single auxiliary polynomial $Q_n$. The next lemma records the exact recursion behind that reduction.

\begin{lemma}[Exact recursion for the stationary family]\label{lem:exact-recursion}
Let $m\ge10$, and let
\[
        u=(0,-a,-a,a,a,-a,-a,0,a,u_{10},u_{11},\ldots,u_m)\in\Z^m.
\]
Define
\begin{align*}
Q_n(c):={}&c^{13}+(n-1)c^{11}+nc^{10}-(n-1)c^9+2(n-1)c^7\notag\\
&-(n-1)c^5+c^3+(n-1)c^2-(n-1).
\end{align*}
Then
\[
        W_{\Rop(u)}(c)+W_u(c)=2c^{m-10}Q_n(c).
\]
For $c\ne0$, equivalently, if $\widW_u(c):=c^{-|u|}W_u(c)$, then
\begin{equation}\label{eq:recursion-normalized}
        \widW_{\Rop(u)}(c)=2c^{-13}Q_n(c)-c^{-3}\widW_u(c).
\end{equation}
\end{lemma}

\begin{proof}
A direct expansion gives
\begin{align*}
W_u(c)={}&2c^m+ac^{m-2}+ac^{m-3}-ac^{m-4}-ac^{m-5}\\
&\quad +ac^{m-6}+ac^{m-7}-ac^{m-9}
      -\sum_{j=10}^m u_jc^{m-j},
\end{align*}
and
\begin{align*}
W_{\Rop(u)}(c)={}&2c^{m+3}+ac^{m+1}+ac^m-ac^{m-1}-ac^{m-2}\\
&\quad +ac^{m-3}+ac^{m-4}-ac^{m-6}-bc^{m-7}\\
&\quad +ac^{m-8}+ac^{m-9}-(a-u_{10})c^{m-10}
      +\sum_{j=11}^m u_jc^{m-j}.
\end{align*}
Adding the two identities cancels the tail coefficients and yields
\begin{align*}
W_{\Rop(u)}(c)+W_u(c)={}&2c^{m+3}+ac^{m+1}+2nc^m-ac^{m-1}+2ac^{m-3}\\
&\quad -ac^{m-5}+2c^{m-7}+ac^{m-8}-ac^{m-10}.
\end{align*}
This is exactly $2c^{m-10}Q_n(c)$ after substituting $a=2(n-1)$ and $a-b=2$. Dividing by $c^{m+3}$ gives \eqref{eq:recursion-normalized}.
\end{proof}

\begin{corollary}[Closed form for the normalized witness polynomials]\label{cor:closed-form}
Let
\begin{equation}\label{eq:Fn-def}
        F_n(c):=\frac{2Q_n(c)}{c^{10}(c^3+1)}.
\end{equation}
Then, as an identity of rational functions,
\begin{equation}\label{eq:closed-form}
        \widW_{w_k}(c)-F_n(c)=(-c^{-3})^{k-1}\bigl(\widW_{w_1}(c)-F_n(c)\bigr)
\end{equation}
for every $k\ge1$. Equivalently,
\[
        \widW_{w_k}(c)=F_n(c)+(-c^{-3})^{k-1}\bigl(\widW_{w_1}(c)-F_n(c)\bigr).
\]
\end{corollary}

\begin{proof}
The identity $F_n(c)=2c^{-13}Q_n(c)-c^{-3}F_n(c)$ is just the definition of $F_n(c)$ rewritten. Subtracting it from \eqref{eq:recursion-normalized} gives
\[
        \widW_{w_{k+1}}(c)-F_n(c)=-c^{-3}\bigl(\widW_{w_k}(c)-F_n(c)\bigr).
\]
Iterating this relation proves the formula.
\end{proof}

\begin{lemma}[Stationary difference polynomials]\label{lem:stationary-diffs}
For $i=1,2,3,4$ define
\[
        \Delta_i(c):=Z_{d_1^{(i)}}(c)-Z_{w_1}(c).
\]
Then for every $k\ge1$ one has
\[
        Z_{d_k^{(i)}}(c)-Z_{w_k}(c)=(-1)^{k-1}\Delta_i(c).
\]
Equivalently, for $c\ne0$,
\[
        W_{d_k^{(i)}}(c)-W_{w_k}(c)=(-1)^{k-1}c^{-1}\Delta_i(c).
\]
Moreover,
\[
        \Delta_1(c)=cA_n(c),\qquad \Delta_2(c)=-c^2B_n(c),\qquad
        \Delta_3(c)=-cA_n(c),\qquad \Delta_4(c)=c^2B_n(c),
\]
where
\begin{align*}
A_n(c)&:=c^8+(n-1)c^6+(n-2)c^5-(n-1)c^4-2(n-1)c^3+c^2+2(n-1)c+2(n-1),\\
B_n(c)&:=c^7+(n-1)c^5+(n-2)c^4-(n-1)c^3-2(n-1)c^2-c+2(n-1).
\end{align*}
\end{lemma}

\begin{proof}
Apply Lemma~\ref{lem:exact-recursion} to the pair of words $d_k^{(i)}$ and $w_k$. Since both have the same fixed initial block, subtraction gives
\[
        W_{d_{k+1}^{(i)}}(c)-W_{w_{k+1}}(c)=-\bigl(W_{d_k^{(i)}}(c)-W_{w_k}(c)\bigr).
\]
Multiplying by $c$ yields
\[
        Z_{d_{k+1}^{(i)}}(c)-Z_{w_{k+1}}(c)=-\bigl(Z_{d_k^{(i)}}(c)-Z_{w_k}(c)\bigr),
\]
and induction on $k$ proves the first formula. The second is immediate from $Z_u=cW_u$. For $k=1$, a direct expansion from \eqref{eq:base-words} and \eqref{eq:w1} gives the displayed formulas for $\Delta_1,\ldots,\Delta_4$.
\end{proof}

\begin{lemma}[Uniform isolation box]\label{lem:uniform-isolation}
Fix $n\ge3$. Then the box $\B_n$ has the following properties:
\begin{enumerate}[label=(\roman*)]
\item $\B_n\subset\X_n\setminus\Rr$;
\item for each $k\ge1$, the polynomial $W_{w_k}$ has exactly one zero in $\B_n$, counted with multiplicity, and no zero on $\partial\B_n$; denote this necessarily simple zero by $c_{n,k}$;
\item every $c_{n,k}$ satisfies
\begin{equation}\label{eq:cnk-imag-bound}
        \Impart(c_{n,k})>\sqrt n,
        \qquad |c_{n,k}|>\sqrt n.
\end{equation}
\end{enumerate}
\end{lemma}

\begin{proof}
If $c=x+\ii y\in\B_n$, then $y>0$, so $c\notin\Rr$, and $|c|^2=x^2+y^2>n$, so $c\in\C\setminus\clD$. Moreover,
\[
        |c+1|^2=(x+1)^2+y^2<\left(\frac32\right)^2+\left(\sqrt n+\frac18\right)^2
        =n+\frac14\sqrt n+\frac{145}{64}<2n
\]
for every $n\ge3$, while $|c-1|<|c+1|$ because $x>0$. Hence $\B_n\subset\X_n\setminus\Rr$.

For (ii), set $F_n(c)$ as in \eqref{eq:Fn-def}. By Corollary~\ref{cor:closed-form},
\[
        \widW_{w_k}(c)-F_n(c)=(-c^{-3})^{k-1}\bigl(\widW_{w_1}(c)-F_n(c)\bigr).
\]
Since $c^3+1\ne0$ on $\B_n$, the functions $F_n$ and $Q_n$ have the same zeros there. Lemma~\ref{lem:boundary-winding} shows that $Q_n$ has exactly one zero in $\B_n$, counted with multiplicity, and none on $\partial\B_n$. Moreover, Lemma~\ref{lem:rouche-bound} gives
\[
        |\widW_{w_1}(c)-F_n(c)|<|F_n(c)|\qquad(c\in\partial\B_n).
\]
Because $|c|>1$ on $\partial\B_n$, the same inequality holds with $\widW_{w_k}$ in place of $\widW_{w_1}$ for every $k\ge1$. Rouch\'e's theorem therefore implies that $\widW_{w_k}$, hence also $W_{w_k}$, has the same number of zeros in $\B_n$, counted with multiplicity, as $F_n$, and therefore as $Q_n$, namely one. Denote this zero by $c_{n,k}$. Since the multiplicity count is one, the zero is simple. The inequalities in \eqref{eq:cnk-imag-bound} are immediate from $c_{n,k}\in\B_n$.
\end{proof}

\begin{theorem}[Stationary transport of the vacancy]\label{thm:stationary}
For every $n\ge3$ and every odd integer $k\ge1$:
\begin{enumerate}[label=(\roman*)]
\item the quintuple $(d_k^{(1)},d_k^{(2)},d_k^{(3)},d_k^{(4)};w_k)$ is a ghost parallelogram of depth $m_k$;
\item there exist real-analytic arcs
\[
        \gamma_{1,k}\subset\partial_{\X_n\setminus\Rr}\U_k^{(1)},\quad
        \gamma_{2,k}\subset\partial_{\X_n\setminus\Rr}\U_k^{(2)},\quad
        \gamma_{3,k}\subset\partial_{\X_n\setminus\Rr}\U_k^{(3)},\quad
        \gamma_{4,k}\subset\partial_{\X_n\setminus\Rr}\U_k^{(4)}
\]
meeting cyclically and transversely such that
\[
        \Gamma_{n,k}:=\gamma_{1,k}\cup\gamma_{2,k}\cup\gamma_{3,k}\cup\gamma_{4,k}
\]
is a Jordan curve contained in $\M_n$;
\item the witness parameter $c_{n,k}$ lies in the bounded component $\Omega_{n,k}$ of $\C\setminus\Gamma_{n,k}$ and satisfies $c_{n,k}\notin\M_n$; hence $(\Gamma_{n,k},c_{n,k})$ is a loop witness;
\item if $\ell>j\ge1$ are odd integers, then
\[
        c_{n,\ell}\notin\Omega_{n,j}.
\]
In particular, the holes produced by the odd-level family $(\Gamma_{n,k},c_{n,k})_{k\ge1,\ k\ {\rm odd}}$ are pairwise distinct.
\end{enumerate}
\end{theorem}

\begin{proof}
Part (i) is Corollary~\ref{cor:transported-ghosts}.

For (ii), the base case is Proposition~\ref{prop:base-loop}. For odd $k>1$, Lemma~\ref{lem:stationary-diffs} rewrites the four level-$k$ chart equations in the stationary form recorded in Appendix~\ref{app:stationary-formulas}. The local branch selection is certified in Proposition~\ref{prop:branch-selection}; the consecutive corner systems, their transversality, and the exclusion of nonconsecutive intersections are provided by Lemma~\ref{lem:corner-systems}; and Lemma~\ref{lem:local-sign-region} identifies the resulting sign region with the bounded complementary component. Thus the four arcs form a Jordan curve $\Gamma_{n,k}$, and Lemma~\ref{lem:chart-inclusion} gives $\Gamma_{n,k}\subset\M_n$.

For (iii), Lemma~\ref{lem:uniform-isolation} gives the unique witness parameter $c_{n,k}$. By Lemma~\ref{lem:local-sign-region}, the bounded component $\Omega_{n,k}$ of $\C\setminus\Gamma_{n,k}$ is exactly the corresponding sign region. The exact sign pattern at $c_{n,k}$ is provided by Lemma~\ref{lem:witness-sign}, so $c_{n,k}\in\Omega_{n,k}$. By Lemma~\ref{lem:extinction}, one has $T_{m_k-1}(c_{n,k})=\varnothing$, so Proposition~\ref{prop:inverse-tree} gives $c_{n,k}\notin\M_n$. Thus $(\Gamma_{n,k},c_{n,k})$ is a loop witness.

Part (iv) is the odd-level no-return conclusion of Lemma~\ref{lem:witness-sign}, whose proof uses Proposition~\ref{prop:no-return} with $d=\ell-j$ even.
\end{proof}

\section{Proof of the main theorem and renormalization consequences}\label{sec:proof-main}

\begin{proof}[Proof of Theorem~\ref{thm:main}]
For $n\ge3$ and odd $k\ge1$, let $H_{n,k}$ be the connected component of $(\C\setminus\clD)\setminus\M_n$ containing $c_{n,k}$. By Theorem~\ref{thm:stationary} and Lemma~\ref{lem:loop-hole}, each $H_{n,k}$ is a hole and $H_{n,k}\subset\Omega_{n,k}$. If $\ell>j$ are odd, then Theorem~\ref{thm:stationary}(iv) gives $c_{n,\ell}\notin\Omega_{n,j}$, hence $c_{n,\ell}\notin H_{n,j}$. Since $c_{n,\ell}\in H_{n,\ell}$, the holes $H_{n,\ell}$ and $H_{n,j}$ are distinct. Thus $\M_n$ has infinitely many holes for every $n\ge3$.

For $n=2$, the introduction identifies $\M_2$, after the inversion $\lambda=1/c$, with the connectedness locus studied by Calegari--Koch--Walker. They prove that this locus has infinitely many holes~\cite{CalegariKochWalker2017}. Hence $\M_2$ also has infinitely many holes.
\end{proof}

\begin{corollary}[Canonical algebraic renormalization point]\label{cor:renorm-point}
For every $n\ge3$ there exists a unique simple zero $\xi_n\in\B_n$ such that
\[
        Q_n(\xi_n)=0.
\]
The witness sequence satisfies
\[
        c_{n,k}\longrightarrow \xi_n\qquad(k\to\infty).
\]
In particular,
\[
        \xi_n\in\partial\M_n\subset\M_n,
\]
every neighborhood of $\xi_n$ meets infinitely many pairwise distinct holes of $\M_n$, and $\xi_n$ is algebraic of degree at most $13$. We call $\xi_n$ the renormalization point of the stationary family.
\end{corollary}

\begin{proof}
Fix $n\ge3$ and set $F_n(c)$ as in \eqref{eq:Fn-def}. By Lemma~\ref{lem:boundary-winding}, the polynomial $Q_n$ has exactly one zero in $\B_n$, counted with multiplicity, and no zero on $\partial\B_n$. Denote that zero by $\xi_n$. Because the count is $1$, the zero is simple.

By Corollary~\ref{cor:closed-form},
\[
        \widW_{w_k}(c)-F_n(c)=(-c^{-3})^{k-1}\bigl(\widW_{w_1}(c)-F_n(c)\bigr).
\]
Hence, for every $c\in\B_n$,
\[
        |\widW_{w_k}(c)-F_n(c)|\le |c|^{-3(k-1)}\sup_{z\in\B_n}|\widW_{w_1}(z)-F_n(z)|.
\]
Since $|c|>\sqrt n>1$ on $\B_n$, it follows that $\widW_{w_k}\to F_n$ uniformly on $\B_n$.

Now let $c_*$ be any subsequential limit of $(c_{n,k})_{k\ge1}$. By Lemma~\ref{lem:uniform-isolation}, all terms lie in the compact set $\ol{\B_n}$, so such a limit exists. Choose a subsequence $c_{n,k_j}\to c_*$. Since $\widW_{w_{k_j}}(c_{n,k_j})=0$, the uniform convergence gives
\[
        |F_n(c_{n,k_j})|=|F_n(c_{n,k_j})-\widW_{w_{k_j}}(c_{n,k_j})|\le\sup_{z\in\B_n}|\widW_{w_{k_j}}(z)-F_n(z)|\longrightarrow0.
\]
Therefore $F_n(c_*)=0$. Because the denominator of $F_n$ does not vanish on $\B_n$, the functions $F_n$ and $Q_n$ have the same zeros there, so $Q_n(c_*)=0$. By uniqueness of the zero in $\B_n$ and the absence of boundary zeros, one must have $c_*=
\xi_n$. Thus every convergent subsequence has the same limit $\xi_n$, and therefore the whole sequence $(c_{n,k})_{k\ge1}$ converges to $\xi_n$.

For each odd $k\ge1$, let $H_{n,k}$ be the hole containing $c_{n,k}$. By the proof of Theorem~\ref{thm:main}, the holes $H_{n,k}$ with $k$ odd are pairwise distinct. Every neighborhood of $\xi_n$ contains all but finitely many odd-level points $c_{n,k}$, hence meets infinitely many of these holes.

We claim that $\xi_n\in\M_n$. Otherwise $\xi_n$ would lie in a connected component $C$ of $(\C\setminus\clD)\setminus\M_n$. Because $(\C\setminus\clD)\setminus\M_n$ is open in $\C\setminus\clD$, some neighborhood $V$ of $\xi_n$ in $\C\setminus\clD$ satisfies $V\subset C$. For all sufficiently large odd $k$, one has $c_{n,k}\in V\subset C$. Since $H_{n,k}$ is the connected component of $(\C\setminus\clD)\setminus\M_n$ containing $c_{n,k}$, it follows that $H_{n,k}=C$ for all large odd $k$, contradicting the pairwise distinctness of the odd-level holes $H_{n,k}$. Thus $\xi_n\in\M_n$. Because every neighborhood of $\xi_n$ meets $(\C\setminus\clD)\setminus\M_n$, we obtain $\xi_n\in\partial\M_n$. Since $\xi_n$ is a zero of the explicit degree-$13$ polynomial $Q_n$, it is algebraic of degree at most $13$.
\end{proof}

\begin{remark}[Quantitative convergence]\label{rem:quant-conv}
The same argument yields a geometric rate. Indeed, Corollary~\ref{cor:closed-form} gives
\[
        \sup_{c\in\B_n}|\widW_{w_k}(c)-F_n(c)|
        \le n^{-3(k-1)/2}\sup_{c\in\B_n}|\widW_{w_1}(c)-F_n(c)|.
\]
Because $\xi_n$ is a simple zero of $F_n$, one can choose a closed disk $K_n\Subset\B_n$ centered at $\xi_n$ and a constant $m_n>0$ such that
\[
        F_n(c)=(c-\xi_n)G_n(c),\qquad |G_n(c)|\ge m_n\qquad(c\in K_n).
\]
Since $c_{n,k}\to\xi_n$, all sufficiently large witness points lie in $K_n$, and for those indices
\[
        |c_{n,k}-\xi_n|\le\frac1{m_n}|F_n(c_{n,k})|
        \le\frac1{m_n}\sup_{c\in\B_n}|\widW_{w_k}(c)-F_n(c)|.
\]
Thus $|c_{n,k}-\xi_n|=O(n^{-3k/2})$ as $k\to\infty$ for each fixed $n$.
\end{remark}

\begin{remark}[Comparison with the $n=2$ renormalization picture]
For $n=2$, Calegari--Koch--Walker identify a renormalization point for $\M_2$ in their reciprocal parameterization and prove that infinitely many holes accumulate there~\cite{CalegariKochWalker2017}. For $n\ge3$, Corollary~\ref{cor:renorm-point} gives the analogous conclusion for the stationary family constructed here, but now with a canonical algebraic limit point: the whole witness sequence converges to the unique zero $\xi_n$ of $Q_n$ in $\B_n$. The mechanisms are different. In~\cite{CalegariKochWalker2017} the renormalization picture is tied to the two-map geometry of $\M_2$; here it comes from stationary transport of a four-chart ghost vacancy, uniform in $n$.
\end{remark}

\begin{figure}[htbp]
\centering
\includegraphics[width=0.98\textwidth]{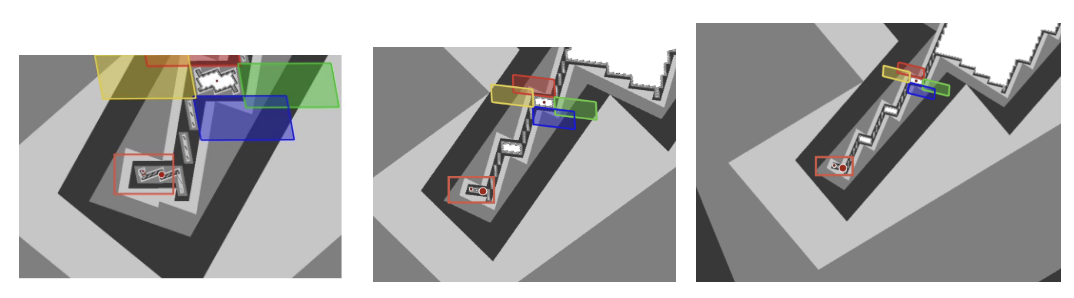}
\caption{Finite-depth inverse-tree renderings of the first stationary hole for $n=3,4,5$. The grayscale parameter-space layer is generated by the finite-capture inverse-iteration procedure described in Appendix~\ref{app:certified-rendering}; the colored overlays mark the chart levels used in the four-chart construction.}
\label{fig:n345-comparison}
\end{figure}

Figure~\ref{fig:n345-comparison} records the same local construction for the first three values $n=3,4,5$. The side-by-side rendering shows how the first visible hole changes with the alphabet size, while the chart-level colors are kept fixed across the three panels. The theorem above proves one stationary family for each $n\ge3$; a systematic classification of further local configurations is left open.

\appendix

\section{Imported inputs from the companion paper}\label{app:imported-inputs}

This appendix records the finite-capture inputs imported from~\cite{EspiguleJuher2026} in the exact form used in the present paper. The restricted-root closure argument is included in Proposition~\ref{prop:restricted-root}; the marked-point criterion and the canonical trap--enclosure framework are restated here with enough detail to make their use transparent.

\subsection{Restricted roots and the marked point}\label{app:restricted-roots}
For $|c|>1$, the difference attractor of the alphabet $A_N$ has the radix description
\[
        E(c,N)=\left\{\sum_{j=0}^{\infty}e_jc^{-j}:e_j\in A_N\right\}.
\]
Hence
\[
        2c\in E(c,N)
        \Longleftrightarrow
        2c=\sum_{j=0}^{\infty}e_jc^{-j}\text{ for some }e_j\in A_N
        \Longleftrightarrow
        1+\sum_{k=1}^{\infty}d_kc^{-k}=0
\]
with $d_k=-e_{k-1}/2\in D_n$. This is the reciprocal-series form appearing in Proposition~\ref{prop:restricted-root}. Every finite admissible word gives an eventually zero digit sequence and therefore a point of $\Rs_n$. Conversely, if the reciprocal series vanishes at $c$, then its truncations define admissible monic polynomials whose roots converge back to $c$ by the Rouch\'e argument recorded in Proposition~\ref{prop:restricted-root}. This is the closure mechanism behind the restricted-root formulation.

\subsection{\texorpdfstring{Why connectedness is detected by $2c$}{Why connectedness is detected by 2c}}\label{app:marked-point}
Let $E_0:=E(c,n)$. The first-level pieces are the translates
\[
        f_t(E_0)=t+\frac1cE_0\qquad(t\in A_n).
\]
For $s,t\in A_n$ one has
\[
        f_t(E_0)\cap f_s(E_0)\ne\varnothing
        \Longleftrightarrow
        c(t-s)\in E_0-E_0.
\]
Since $A_n$ is an arithmetic progression with common gap $2$, the corresponding overlap criterion can be expressed entirely in terms of the difference attractor of the same collinear family. In the present normalization, the criterion of~\cite[Proposition 2.1]{EspiguleJuher2026} becomes
\[
        c\in\M_n\quad\Longleftrightarrow\quad 2c\in E_0-E_0.
\]
Because $E_0-E_0$ is precisely the difference attractor $E(c,N)$, this yields Proposition~\ref{prop:marked-point}.

\subsection{Canonical trap and enclosure}\label{app:trap}
For $c=x+\ii y\in\X_n\setminus\Rr$, the companion paper constructs two canonical parallelograms adapted to the linear dynamics of multiplication by $c$: an open trap $\Ctrap(c)$ and a closed enclosure $\Enc(c,N)$. In the normalization used here, the relevant trap half-widths are
\[
        S(c,N)=\frac{Ny}{|c|},\qquad
        V(c,N)=\frac{(N-2x)y}{|c|^2},
\]
exactly as recorded later in Appendix~\ref{app:stationary-formulas}. The enclosure is forward invariant under the Hutchinson operator, so it contains the whole difference attractor. The trap is strictly covered by the union of its first-generation images and is therefore contained in the interior of the attractor. Because the digit set $A_N$ is an arithmetic progression on the real axis, both verifications reduce to one-dimensional interval-covering checks in the two canonical coordinates. This is the content of~\cite[Definition 3.11, Corollary 3.13, Proposition 3.15, and Theorem 3.16]{EspiguleJuher2026}, restated here as Proposition~\ref{prop:trap-enclosure}.

\section{Explicit stationary formulas}\label{app:stationary-formulas}

For $c=x+\ii y$ with $x>0$ and $y>0$, the canonical trap from Proposition~\ref{prop:trap-enclosure} has half-widths
\[
        S(c,N)=\frac{Ny}{|c|},\qquad V(c,N)=\frac{(N-2x)y}{|c|^2}.
\]
These are trap widths; later in Appendix~\ref{app:exact-verification} we also use the enclosure widths $S_E(c,N)$ and $V_E(c,N)$. Hence the chart condition $Z_u(c)\in\Ctrap(c)$ becomes
\[
        c\in\U_u
        \Longleftrightarrow
        |\Impart Z_u(c)|<\frac{(N-2x)y}{x^2+y^2}
        \quad\text{and}\quad
        |\Impart((x+\ii y)Z_u(x+\ii y))|<Ny
\]
on the upper half-plane. For later use, write these two inequalities as
\begin{align}
V_u^-(x,y)&:=(x^2+y^2)\Impart Z_u(x+\ii y)+(N-2x)y,\label{eq:Vu-minus}\\
V_u^+(x,y)&:=(x^2+y^2)\Impart Z_u(x+\ii y)-(N-2x)y,\label{eq:Vu-plus}\\
S_u^-(x,y)&:=\Impart\bigl((x+\ii y)Z_u(x+\ii y)\bigr)+Ny,\label{eq:Su-minus}\\
S_u^+(x,y)&:=\Impart\bigl((x+\ii y)Z_u(x+\ii y)\bigr)-Ny.\label{eq:Su-plus}
\end{align}
Then, for $y>0$,
\begin{equation}\label{eq:chart-signs}
        c=x+\ii y\in\U_u
        \Longleftrightarrow
        V_u^-(x,y)>0,
        \quad V_u^+(x,y)<0,
        \quad S_u^-(x,y)>0,
        \quad S_u^+(x,y)<0.
\end{equation}
Consequently every boundary component of $\U_u$ on the upper half-plane is contained in one of the four real-algebraic curves
\[
        V_u^-=0,\quad V_u^+=0,
        \quad S_u^-=0,
        \quad S_u^+=0.
\]

For the stationary family, Lemma~\ref{lem:stationary-diffs} gives the exact identities
\begin{equation}\label{eq:stationary-Z}
        Z_{d_k^{(i)}}(c)=Z_{w_k}(c)+(-1)^{k-1}\Delta_i(c)
        \qquad(i=1,2,3,4).
\end{equation}
Thus the only $k$-dependence in the level-$k$ chart equations comes from the single holomorphic function $Z_{w_k}$; the four perturbation terms are the fixed low-degree polynomials $\Delta_1,\ldots,\Delta_4$. In particular, after inserting \eqref{eq:stationary-Z} into \eqref{eq:Vu-minus}--\eqref{eq:Su-plus}, every geometric statement about the stationary family reduces to explicit real polynomials and rational functions in $(x,y)$ together with the algebraic equation defining $c_{n,k}$.

For the four loop branches we use the shorthand
\begin{equation}\label{eq:branch-functions-shorthand}
        F_{1,k}:=V_{d_k^{(1)}}^-,\qquad
        F_{2,k}:=S_{d_k^{(2)}}^-,\qquad
        F_{3,k}:=V_{d_k^{(3)}}^+,
        \qquad F_{4,k}:=S_{d_k^{(4)}}^+.
\end{equation}
The four consecutive corner systems are therefore
\begin{equation}\label{eq:corner-systems}
        F_{1,k}=F_{2,k}=0,
        \quad F_{2,k}=F_{3,k}=0,
        \quad F_{3,k}=F_{4,k}=0,
        \quad F_{4,k}=F_{1,k}=0.
\end{equation}

\section{Finite algebraic verification}\label{app:exact-verification}

The remaining finite subarguments reduce to sign, root-counting, and non-intersection statements for explicit real polynomials and rational functions on compact intervals, rectangles, and algebraic zero-dimensional sets. This appendix records these reductions in a form that can be checked by finite algebraic certificates. Each entry specifies the polynomial family, the domain, the verification method, and the conclusion needed in the main proof. The methods used are Sturm's theorem, subresultants, denominator clearing, Bernstein coefficient positivity, elementary interval estimates, and the displayed stationary identities. The parameter-space figures are finite-depth renderings; their rendering protocol is recalled in Appendix~\ref{app:certified-rendering}.

\subsection{Verification protocol}\label{subsec:verification-protocol}
All verifications used in the proof are restricted to the odd-depth subfamily $k=1,3,5,\ldots$. They use rational arithmetic after replacing $\sqrt n$ by a parameter satisfying the relevant quadratic relation when edge intervals are parameterized. The infinite ranges of $n$ and $k$ are reduced as follows. The witness-box boundary computation depends only on $n$ and is split into the five regimes listed in Table~\ref{tab:itinerary}. The local loop and pruning verifications use the algebraic relation $W_{w_k}(c_{n,k})=0$ together with the stationary identities in Corollary~\ref{cor:closed-form} and Lemma~\ref{lem:stationary-diffs}, so that the parameter $k$ enters only through explicit parity and monomial factors. The no-return verification is made for odd $k$ and positive even $d$ on the algebraic set cut out by the later-witness equation $W_{w_{k+d}}(c)=0$ in $\ol{\B_n}$, not on the whole witness box. Each item below is therefore a finite statement about signs of explicitly displayed polynomials on explicitly displayed intervals, rectangles, or zero-dimensional algebraic sets.

\begin{table}[t]
\centering
\small
\caption{Index of the finite algebraic verifications. Each row gives the verification target, the verification object, and the acceptance rule used in the proof.}
\label{tab:verification-index}
\begin{tabular}{@{}p{0.10\linewidth}p{0.28\linewidth}p{0.30\linewidth}p{0.22\linewidth}@{}}
\toprule
Entry & Proof target & Verification data & Acceptance rule\\
\midrule
H1 & Unique witness-box root and Rouch\'e boundary margin & Edge polynomials $R_{q,n},I_{q,n}$ and cleared margins $H_{q,n}$ on the four edges of $\B_n$ & Sturm variation counts and strict positive margin bounds\\
H2 & Four-chart Jordan loop & Auxiliary chart signs, consecutive corner systems, Jacobians, graph data, and nonconsecutive resultants on $\mathcal N_{n,k}$ & Face signs, nonzero Jacobians, graph-arc criterion, and subresultant exclusion\\
H3 & Witness inclusion and no-return & Cleared signs $\Phi_{j,k,d}$ restricted to the algebraic witness locus $W_{w_{k+d}}=0$ for odd $k$ and even $d>0$ & Root isolation and signed remainders modulo the witness equation\\
H4 & Inverse-tree extinction & Early-prefix pruning, admissible digit intervals, terminal rows, and transported digit equivalences & Interval inclusions between consecutive digits of $A_N$\\
\bottomrule
\end{tabular}
\end{table}

\subsection{Root isolation in the common witness box}\label{subsec:H1}
Verification H1. The inputs are the edge polynomials $R_{q,n}$ and $I_{q,n}$ and the cleared margin polynomials $H_{q,n}$ used in the Rouch\'e step; the outputs are Proposition~\ref{prop:finite-boundary-data} and Lemmas~\ref{lem:boundary-winding} and~\ref{lem:rouche-bound}, hence the unique witness-box root statement recorded in Lemma~\ref{lem:uniform-isolation}.

Write
\[
        I_x:=\left[\frac14,\frac12\right],
        \qquad
        I_y:=\left[\sqrt n,\sqrt n+\frac18\right].
\]
Traversing $\partial\B_n$ positively, parameterize its four edges by
\begin{align*}
        \sigma_b(t)&:=t+\ii\sqrt n &&(t\in I_x),\\
        \sigma_r(t)&:=\frac12+\ii t &&(t\in I_y),\\
        \sigma_t(t)&:=\frac34-t+\ii\left(\sqrt n+\frac18\right) &&(t\in I_x),\\
        \sigma_l(t)&:=\frac14+\ii\left(2\sqrt n+\frac18-t\right) &&(t\in I_y).
\end{align*}
For $q\in\{b,r,t,l\}$ write
\[
        Q_{q,n}(t):=Q_n(\sigma_q(t))=R_{q,n}(t)+\ii I_{q,n}(t),
\]
where $R_{q,n}$ and $I_{q,n}$ are the corresponding one-variable real polynomials obtained by substitution. For the Rouch\'e step we also set the auxiliary numerator polynomial $\Phi_n$ (distinct from the cleared no-return polynomials $\Phi_{j,k,d}$ appearing in Section~\ref{subsec:H3}):
\begin{align}
\Phi_n(c)&:=(c^3+1)W_{w_1}(c)-2c^8Q_n(c)\notag\\
&=-c^{11}-(n-1)c^9-(n-1)c^8+(n-1)c^7+(n-1)c^6+(n-3)c^5\notag\\
&\qquad -(n-1)c^4+(2n-3)c^2.\label{eq:Phi-n-poly}
\end{align}
Thus
\[
        \widW_{w_1}(c)-F_n(c)=\frac{\Phi_n(c)}{c^{18}(c^3+1)}.
\]

For the boundary argument we separate the five arithmetic regimes
\begin{align*}
        \mathcal N_1&:=\{3,4\},&
        \mathcal N_2&:=\{5\},&
        \mathcal N_3&:=\{6,7,8\},\\
        \mathcal N_4&:=\{9,10,\ldots,29\},&
        \mathcal N_5&:=\{n\in\Z:n\ge30\}.
\end{align*}
Along each oriented edge of $\partial\B_n$, the image under $Q_n$ follows a fixed quadrant itinerary inside each of these five ranges. Equivalently, for each fixed $n$ and each edge polynomial $Q_{q,n}=R_{q,n}+\ii I_{q,n}$, the real polynomials $R_{q,n}$ and $I_{q,n}$ have the exact root counts and initial signs encoded by that itinerary. This is the finite one-variable content of the winding computation, and Table~\ref{tab:itinerary} records the corresponding quadrant data.

\begin{table}[t]
\centering
\caption{Quadrant itinerary of the oriented boundary image $Q_n(\partial\B_n)$; equivalently, the exact zero pattern of the real and imaginary edge polynomials $R_{q,n}$ and $I_{q,n}$ on the four oriented edges. Each arrow indicates one simple axis crossing. The top edge is traversed from right to left and the left edge from top to bottom.}
\label{tab:itinerary}
\begin{tabular}{ccccc}
\toprule
range of $n$ & bottom edge & right edge & top edge & left edge\\
\midrule
$\mathcal N_1$ & II & II$\to$III & III$\to$IV$\to$I & I$\to$II\\
$\mathcal N_2$ & II$\to$III$\to$II & II$\to$III & III$\to$IV$\to$I & I$\to$II\\
$\mathcal N_3$ & II$\to$III & III & III$\to$IV$\to$I$\to$II & II\\
$\mathcal N_4$ & II$\to$III & III$\to$IV & IV$\to$I$\to$II & II\\
$\mathcal N_5$ & II$\to$III & III$\to$IV$\to$I & I$\to$II & II\\
\bottomrule
\end{tabular}
\end{table}

For the Rouch\'e step, the cleared margin polynomials
\[
        H_{q,n}(t):=4|\sigma_q(t)|^{16}|Q_{q,n}(t)|^2-|\Phi_n(\sigma_q(t))|^2
        \qquad(q\in\{b,r,t,l\})
\]
are again one-variable real polynomials on the corresponding edge intervals. The concrete lower bounds needed later are summarized in Table~\ref{tab:rouche-margins}.

\begin{table}[t]
\centering
\caption{Explicit positive lower bounds for the cleared Rouch\'e margins $H_{q,n}$ on the four oriented edges of $\partial\B_n$. These are the actual edgewise inequalities used in the Rouch\'e step.}
\label{tab:rouche-margins}
\begin{tabular}{ccccc}
\toprule
range of $n$ & bottom & right & top & left\\
\midrule
$\mathcal N_1$ & $H_{b,n}>10^9$ & $H_{r,n}>10^{10}$ & $H_{t,n}>10^9$ & $H_{l,n}>10^9$\\
$\mathcal N_2$ & $H_{b,n}>10^{13}$ & $H_{r,n}>10^{14}$ & $H_{t,n}>10^{13}$ & $H_{l,n}>10^{13}$\\
$\mathcal N_3$ & $H_{b,n}>10^{15}$ & $H_{r,n}>10^{15}$ & $H_{t,n}>10^{15}$ & $H_{l,n}>10^{15}$\\
$\mathcal N_4$ & $H_{b,n}>10^{18}$ & $H_{r,n}>10^{18}$ & $H_{t,n}>10^{18}$ & $H_{l,n}>10^{19}$\\
$\mathcal N_5$ & $H_{b,n}>10^{28}$ & $H_{r,n}>10^{28}$ & $H_{t,n}>10^{28}$ & $H_{l,n}>10^{29}$\\
\bottomrule
\end{tabular}
\end{table}

If $N_J(P)$ denotes the number of real zeros of a real polynomial $P$ on an interval $J$, then Table~\ref{tab:itinerary} is equivalent to the explicit edgewise counts
\begin{align}
N_{I_x}(R_{b,n})&=0,&
N_{I_x}(I_{b,n})&=\begin{cases}0,&n\in\mathcal N_1,\\2,&n\in\mathcal N_2,\\1,&n\in\mathcal N_3\cup\mathcal N_4\cup\mathcal N_5,
\end{cases}\notag\\
N_{I_y}(R_{r,n})&=\begin{cases}0,&n\in\mathcal N_1\cup\mathcal N_2\cup\mathcal N_3,\\1,&n\in\mathcal N_4\cup\mathcal N_5,
\end{cases}&
N_{I_y}(I_{r,n})&=\begin{cases}1,&n\in\mathcal N_1\cup\mathcal N_2,\\0,&n\in\mathcal N_3\cup\mathcal N_4,\\1,&n\in\mathcal N_5,
\end{cases}\notag\\
N_{I_x}(R_{t,n})&=\begin{cases}1,&n\in\mathcal N_1\cup\mathcal N_2\cup\mathcal N_4\cup\mathcal N_5,\\2,&n\in\mathcal N_3,
\end{cases}&
N_{I_x}(I_{t,n})&=\begin{cases}1,&n\in\mathcal N_1\cup\mathcal N_2\cup\mathcal N_3\cup\mathcal N_4,\\0,&n\in\mathcal N_5,
\end{cases}\notag\\
N_{I_y}(R_{l,n})&=\begin{cases}1,&n\in\mathcal N_1\cup\mathcal N_2,\\0,&n\in\mathcal N_3\cup\mathcal N_4\cup\mathcal N_5,
\end{cases}&
N_{I_y}(I_{l,n})&=0.
\label{eq:edge-counts}
\end{align}
Together with the initial quadrant on each oriented edge, these counts determine the complete boundary itinerary. Likewise Table~\ref{tab:rouche-margins} records the actual Rouch\'e inequalities
\begin{equation}\label{eq:margins}
        H_{b,n}(t)>M_b(n),\qquad H_{r,n}(t)>M_r(n),\qquad H_{t,n}(t)>M_t(n),\qquad H_{l,n}(t)>M_l(n)
\end{equation}
on the corresponding edge intervals, where the regime-wise lower bounds $M_q(n)$ are exactly the entries of Table~\ref{tab:rouche-margins}.

\begin{proposition}[Finite witness-box boundary data]\label{prop:finite-boundary-data}
For every $n\ge3$ the witness-box boundary data satisfy the following exact statements.
\begin{enumerate}[label=(\roman*)]
\item The edge polynomials $R_{q,n}$ and $I_{q,n}$ satisfy the exact root counts in \eqref{eq:edge-counts}. Equivalently, $Q_n(\partial\B_n)$ follows the quadrant itinerary in Table~\ref{tab:itinerary}.
\item The cleared margin polynomials $H_{q,n}$ satisfy the positive lower bounds in \eqref{eq:margins}. Consequently the edgewise Rouch\'e inequalities hold on the four boundary edges.
\end{enumerate}
\end{proposition}

\begin{proof}
For each of the finitely many low regimes, apply Sturm's theorem to the displayed edge and margin polynomials on the stated compact intervals. The variation drops are exactly those in \eqref{eq:edge-counts}; the initial signs give Table~\ref{tab:itinerary}. For the regime $n\ge30$, write $n=30+s$ with $s\ge0$ and apply the same Sturm calculation after clearing the quadratic relation defining $\sqrt n$; the signed remainders have non-negative coefficients in $s$ and at least one strictly positive coefficient on every relevant interval endpoint. This gives the same variation drops and the lower bounds in Table~\ref{tab:rouche-margins}. Thus both assertions follow from Sturm's theorem and elementary coefficient positivity.
\end{proof}

\begin{lemma}[Boundary winding for $Q_n$]\label{lem:boundary-winding}
For every $n\ge3$, the oriented image $Q_n(\partial\B_n)$ winds once around the origin. Consequently, $Q_n$ has exactly one zero in $\B_n$, counted with multiplicity, and no zero on $\partial\B_n$.
\end{lemma}

\begin{proof}
By Proposition~\ref{prop:finite-boundary-data}(i), the exact root counts of the edge polynomials are those in \eqref{eq:edge-counts}, hence the oriented boundary image follows the quadrant itinerary recorded in Table~\ref{tab:itinerary}. In each of the five regimes the itinerary has total winding number one and avoids the origin. The argument principle gives exactly one zero of $Q_n$ in $\B_n$, counted with multiplicity, and no boundary zero.
\end{proof}

\begin{lemma}[Uniform Rouch\'e bound]\label{lem:rouche-bound}
For every $n\ge3$ and every $c\in\partial\B_n$,
\[
        |\widW_{w_1}(c)-F_n(c)|<|F_n(c)|.
\]
Equivalently,
\begin{equation}\label{eq:rouche-clear}
        |\Phi_n(c)|<2|c|^8|Q_n(c)|\qquad(c\in\partial\B_n),
\end{equation}
where $\Phi_n$ is the explicit degree-$11$ polynomial from \eqref{eq:Phi-n-poly}. Hence, by Corollary~\ref{cor:closed-form}, the same boundary inequality holds with $\widW_{w_k}$ in place of $\widW_{w_1}$ for every $k\ge1$.
\end{lemma}

\begin{proof}
Because $c\ne0$ and $c^3+1\ne0$ on $\partial\B_n$, the displayed inequality is equivalent to \eqref{eq:rouche-clear}. On each edge $c=\sigma_q(t)$, the square of \eqref{eq:rouche-clear} is exactly the assertion $H_{q,n}(t)>0$. By Proposition~\ref{prop:finite-boundary-data}(ii), the four pointwise edge inequalities are the bounds in \eqref{eq:margins}. Since the lower bounds are positive, they give \eqref{eq:rouche-clear} on every edge and hence on the whole boundary. The final sentence follows from Corollary~\ref{cor:closed-form} exactly as in the proof of Lemma~\ref{lem:uniform-isolation}.
\end{proof}

\subsection{Loop branches and corner systems}\label{subsec:H2}
Verification H2. The inputs are the four branch equations, their one-variable face restrictions on the corner rectangles, the four consecutive Jacobians, and the two nonconsecutive exclusion polynomials; the outputs are Lemmas~\ref{lem:explicit-branch-functions}, \ref{lem:corner-systems}, and~\ref{lem:local-sign-region} and Proposition~\ref{prop:branch-selection}.

For the consecutive corner systems \eqref{eq:corner-systems}, write
\begin{align*}
        J_{12,k}&:=\det\nabla(F_{1,k},F_{2,k}),&
        J_{23,k}&:=\det\nabla(F_{2,k},F_{3,k}),\\
        J_{34,k}&:=\det\nabla(F_{3,k},F_{4,k}),&
        J_{41,k}&:=\det\nabla(F_{4,k},F_{1,k}).
\end{align*}
At the witness root $c_{n,k}=x_k+\ii y_k$, let
\[
        L_{n,k}:=Z'_{w_k}(c_{n,k})=c_{n,k}W'_{w_k}(c_{n,k})\ne0,
\]
which is nonzero because Lemma~\ref{lem:uniform-isolation} shows that $W_{w_k}$ has exactly one zero in $\B_n$ counted with multiplicity, so the zero $c_{n,k}$ is simple. We use the affine root coordinate
\[
        \zeta=u+\ii v:=L_{n,k}(c-c_{n,k})
\]
around $c_{n,k}$, and we write
\[
        \widehat F_{i,k}(\zeta):=F_{i,k}(c_{n,k}+L_{n,k}^{-1}\zeta)\qquad(i=1,2,3,4).
\]
Set
\begin{align}
\alpha_{n,k}&:=-(-1)^{k-1}\Impart(c_{n,k}A_n(c_{n,k}))-\frac{(N-2x_k)y_k}{x_k^2+y_k^2},\label{eq:alpha}\\
\beta_{n,k}&:=\frac{(-1)^{k-1}\Impart(c_{n,k}^3B_n(c_{n,k}))-Ny_k}{y_k},\label{eq:beta}\\
 u_{12}(n,k)&:=\beta_{n,k}-\frac{x_k}{y_k}\alpha_{n,k},
 \qquad
 u_{23}(n,k):=\beta_{n,k}+\frac{x_k}{y_k}\alpha_{n,k}.
 \label{eq:u12u23}
\end{align}
Let $\eta:=1/20$, and define the four affine Miranda rectangles in the $\zeta$-plane by
\begin{align*}
\widehat R_{12}(n,k)&:=[u_{12}(n,k)-\eta,u_{12}(n,k)+\eta]\times[\alpha_{n,k}-\eta,\alpha_{n,k}+\eta],\\
\widehat R_{23}(n,k)&:=[u_{23}(n,k)-\eta,u_{23}(n,k)+\eta]\times[-\alpha_{n,k}-\eta,-\alpha_{n,k}+\eta],\\
\widehat R_{34}(n,k)&:=[-u_{12}(n,k)-\eta,-u_{12}(n,k)+\eta]\times[-\alpha_{n,k}-\eta,-\alpha_{n,k}+\eta],\\
\widehat R_{41}(n,k)&:=[-u_{23}(n,k)-\eta,-u_{23}(n,k)+\eta]\times[\alpha_{n,k}-\eta,\alpha_{n,k}+\eta].
\end{align*}
We also write
\begin{align*}
U_{n,k}^{\rm loc}&:=\max\{|u_{12}(n,k)|,|u_{23}(n,k)|\}+2\eta,\\
V_{n,k}^{\rm loc}&:=|\alpha_{n,k}|+2\eta,\\
\widehat{\mathcal N}_{n,k}&:=[-U_{n,k}^{\rm loc},U_{n,k}^{\rm loc}]\times[-V_{n,k}^{\rm loc},V_{n,k}^{\rm loc}],\\
\mathcal N_{n,k}&:=c_{n,k}+L_{n,k}^{-1}\widehat{\mathcal N}_{n,k}.
\end{align*}

\begin{lemma}[Explicit branch functions]\label{lem:explicit-branch-functions}
Let $c=x+\ii y$ with $y>0$. Using \eqref{eq:stationary-Z} and the formulas for $\Delta_i$, the four branch functions from \eqref{eq:branch-functions-shorthand} are
\begin{align}
F_{1,k}(c)&=(x^2+y^2)\Impart Z_{w_k}(c)+(-1)^{k-1}(x^2+y^2)\Impart(cA_n(c))+(N-2x)y,\label{eq:F1}\\
F_{2,k}(c)&=\Impart(cZ_{w_k}(c))-(-1)^{k-1}\Impart(c^3B_n(c))+Ny,\label{eq:F2}\\
F_{3,k}(c)&=(x^2+y^2)\Impart Z_{w_k}(c)-(-1)^{k-1}(x^2+y^2)\Impart(cA_n(c))-(N-2x)y,\label{eq:F3}\\
F_{4,k}(c)&=\Impart(cZ_{w_k}(c))+(-1)^{k-1}\Impart(c^3B_n(c))-Ny.\label{eq:F4}
\end{align}
\end{lemma}

\begin{proof}
Substituting
\[
\begin{aligned}
Z_{d_k^{(1)}}&=Z_{w_k}+(-1)^{k-1}cA_n,&
Z_{d_k^{(2)}}&=Z_{w_k}-(-1)^{k-1}c^2B_n,\\
Z_{d_k^{(3)}}&=Z_{w_k}-(-1)^{k-1}cA_n,&
Z_{d_k^{(4)}}&=Z_{w_k}+(-1)^{k-1}c^2B_n
\end{aligned}
\]
into \eqref{eq:Vu-minus}--\eqref{eq:Su-plus} gives the formulas directly.
\end{proof}

\begin{proposition}[Certified branch selection on the local box]\label{prop:branch-selection}
For every $n\ge3$ and every odd $k\ge1$, the compact local parameter box $\mathcal N_{n,k}$ satisfies $\mathcal N_{n,k}\Subset\X_n\setminus\Rr$, contains the four corner rectangles from Lemma~\ref{lem:corner-systems}, and has the following property. On $\mathcal N_{n,k}$, the only boundary branches of
\[
        \partial_{\X_n\setminus\Rr}\U_k^{(1)},\quad
        \partial_{\X_n\setminus\Rr}\U_k^{(2)},\quad
        \partial_{\X_n\setminus\Rr}\U_k^{(3)},\quad
        \partial_{\X_n\setminus\Rr}\U_k^{(4)}
\]
that enter the local loop geometry are exactly the zero sets
\[
        F_{1,k}=0,
        \quad F_{2,k}=0,
        \quad F_{3,k}=0,
        \quad F_{4,k}=0.
\]
More precisely, on $\mathcal N_{n,k}$ the remaining chart inequalities have the fixed signs
\[
\begin{array}{lll}
S_{d_k^{(1)}}^->0,& S_{d_k^{(1)}}^+<0,& V_{d_k^{(1)}}^+<0,\\
V_{d_k^{(2)}}^->0,& V_{d_k^{(2)}}^+<0,& S_{d_k^{(2)}}^+<0,\\
V_{d_k^{(3)}}^->0,& S_{d_k^{(3)}}^->0,& S_{d_k^{(3)}}^+<0,\\
V_{d_k^{(4)}}^->0,& V_{d_k^{(4)}}^+<0,& S_{d_k^{(4)}}^->0.
\end{array}
\]
In addition, each selected zero set $F_{i,k}=0$ inside $\mathcal N_{n,k}$ is a single nonsingular analytic arc joining its two adjacent corner rectangles, it meets no face of $\partial\mathcal N_{n,k}$ except through those corner rectangles, and no additional component of $F_{i,k}=0$ lies in the local sign region defined in Lemma~\ref{lem:local-sign-region}.
\end{proposition}

\begin{proof}
After substituting \eqref{eq:stationary-Z} and the explicit formulas for $\Delta_i$ into the twelve auxiliary boundary functions not chosen in \eqref{eq:branch-functions-shorthand}, one obtains explicit real polynomial/rational expressions in the affine coordinate $\zeta$ with coefficients in the algebraic field generated by $c_{n,k}$. Verification H2 reduces these expressions modulo $W_{w_k}(c_{n,k})=0$ and checks their signs on the box $\widehat{\mathcal N}_{n,k}$ by Sturm chains on faces and interval subresultant arguments in the interior. The same verification proves the nonsingularity and graph structure of the four selected zero sets: after possibly interchanging the $u$- and $v$-coordinates on two arcs, the relevant partial derivative is bounded away from zero, the corresponding one-variable edge restrictions have exactly the endpoint zeros prescribed by Lemma~\ref{lem:corner-systems}, and the signs on $\partial\widehat{\mathcal N}_{n,k}$ exclude escape and re-entry. Equivalently, no auxiliary boundary branch enters the local box, and each selected branch is the intended single analytic arc. This proves the stated branch-selection and connectedness claims.
\end{proof}

\begin{lemma}[Corner systems and transversality]\label{lem:corner-systems}
For every $n\ge3$ and every odd $k\ge1$, put
\[
\begin{gathered}
\eps_k:=(-1)^{k-1},\qquad
U_k:=(x^2+y^2)\Impart Z_{w_k}(c),\qquad
V_k:=\Impart(cZ_{w_k}(c)),\\
A_k:=(x^2+y^2)\Impart(cA_n(c)),\qquad
B_k:=\Impart(c^3B_n(c)).
\end{gathered}
\]
Then the consecutive corner systems can be written explicitly as
\[
\begin{aligned}
F_{1,k}=F_{2,k}=0&\Longleftrightarrow
\begin{cases}
U_k=-\eps_k A_k-(N-2x)y,\\
V_k=\eps_k B_k-Ny,
\end{cases}\\[1ex]
F_{2,k}=F_{3,k}=0&\Longleftrightarrow
\begin{cases}
U_k=\eps_k A_k+(N-2x)y,\\
V_k=\eps_k B_k-Ny,
\end{cases}\\[1ex]
F_{3,k}=F_{4,k}=0&\Longleftrightarrow
\begin{cases}
U_k=\eps_k A_k+(N-2x)y,\\
V_k=-\eps_k B_k+Ny,
\end{cases}\\[1ex]
F_{4,k}=F_{1,k}=0&\Longleftrightarrow
\begin{cases}
U_k=-\eps_k A_k-(N-2x)y,\\
V_k=-\eps_k B_k+Ny.
\end{cases}
\end{aligned}
\]
Moreover, with the affine root coordinate and rectangles above, the following explicit edge-sign conditions hold.
\begin{enumerate}[label=(\roman*)]
\item On $\widehat R_{12}(n,k)$, one has
\[
\widehat F_{1,k}(u+\ii(\alpha_{n,k}-\eta))<0<\widehat F_{1,k}(u+\ii(\alpha_{n,k}+\eta))
\]
for $u\in[u_{12}(n,k)-\eta,u_{12}(n,k)+\eta]$, and
\[
\widehat F_{2,k}(u_{12}(n,k)-\eta+\ii v)<0<\widehat F_{2,k}(u_{12}(n,k)+\eta+\ii v)
\]
for $v\in[\alpha_{n,k}-\eta,\alpha_{n,k}+\eta]$.
\item The analogous edge-sign patterns hold on $\widehat R_{23},\widehat R_{34},\widehat R_{41}$ with the cyclic pairs $(F_2,F_3)$, $(F_3,F_4)$, and $(F_4,F_1)$ and the intervals specified above.
\item The transformed Jacobians
\begin{align*}
        \widehat J_{12,k}&:=\det\nabla(\widehat F_{1,k},\widehat F_{2,k}),&
        \widehat J_{23,k}&:=\det\nabla(\widehat F_{2,k},\widehat F_{3,k}),\\
        \widehat J_{34,k}&:=\det\nabla(\widehat F_{3,k},\widehat F_{4,k}),&
        \widehat J_{41,k}&:=\det\nabla(\widehat F_{4,k},\widehat F_{1,k})
\end{align*}
do not vanish on their respective rectangles.
\item The nonconsecutive systems $\widehat F_{1,k}=\widehat F_{3,k}=0$ and $\widehat F_{2,k}=\widehat F_{4,k}=0$ have no solution in $\widehat{\mathcal N}_{n,k}$.
\end{enumerate}
Consequently each consecutive system has a unique solution in its designated rectangle, the corresponding boundary branches meet transversely there, and there are no nonconsecutive intersections in $\mathcal N_{n,k}$.
\end{lemma}

\begin{proof}
The four displayed equivalences are obtained by substituting \eqref{eq:F1}--\eqref{eq:F4} into the systems in \eqref{eq:corner-systems}. The remaining assertions are finite certificate checks on the explicit rectangles. The one-variable face restrictions are evaluated by Sturm chains; the transformed Jacobians are bounded away from zero on the four rectangles; and the two nonconsecutive systems are eliminated by subresultants on $\widehat{\mathcal N}_{n,k}$. These checks give the face-sign patterns in (i)--(ii), the Jacobian nonvanishing in (iii), and the nonconsecutive exclusions in (iv). The Poincar\'e--Miranda argument with nonzero Jacobian then yields one transverse intersection for each consecutive pair and no others in the local box.
\end{proof}

\begin{lemma}[Local sign region]\label{lem:local-sign-region}
For $n\ge3$ and odd $k\ge1$, let
\begin{equation}\label{eq:Omega-def}
        \Omega_{n,k}^{\rm loc}:=\left\{c\in\mathcal N_{n,k}:
        \begin{array}{l}
        (-1)^{k-1}F_{1,k}(c)<0,
        \quad (-1)^{k-1}F_{2,k}(c)<0,\\
        (-1)^{k-1}F_{3,k}(c)>0,
        \quad (-1)^{k-1}F_{4,k}(c)>0
        \end{array}\right\}.
\end{equation}
Then the four local branches cut out by Proposition~\ref{prop:branch-selection} and Lemma~\ref{lem:corner-systems} form a Jordan curve $\Gamma_{n,k}$, and $\Omega_{n,k}^{\rm loc}$ is its bounded complementary component inside $\mathcal N_{n,k}$. Moreover, the same set is the bounded component $\Omega_{n,k}$ of $\C\setminus\Gamma_{n,k}$.
\end{lemma}

\begin{proof}
By Lemma~\ref{lem:corner-systems}, the four branches meet in cyclic order and have no nonconsecutive intersections. By Proposition~\ref{prop:branch-selection}, each selected zero set is a single analytic arc between its two consecutive corner points, no selected arc exits and re-enters the local box, and no auxiliary chart boundary branch enters the local box. Therefore the union is a simple closed curve. The sign pattern in \eqref{eq:Omega-def} chooses the side of each branch facing the center of the quadrilateral. The certified boundary signs on $\partial\mathcal N_{n,k}$ exclude any connection from this sign region to the exterior of the local box; hence its boundary in the plane is exactly the union of the four arcs. Thus $\Omega_{n,k}^{\rm loc}$ is the bounded complementary component, which we denote by $\Omega_{n,k}$.
\end{proof}

\subsection{Witness inclusion and no return}\label{subsec:H3}
Verification H3. The inputs are the exact witness evaluations \eqref{eq:witness-eval1}--\eqref{eq:witness-eval4}, the later-root identity \eqref{eq:later-root}, and the cleared numerator polynomials $\Phi_{j,k,d}$; the outputs are Proposition~\ref{prop:no-return} and Lemma~\ref{lem:witness-sign}.

For later use note first that, at the witness parameter $c_{n,k}$, the relation $Z_{w_k}(c_{n,k})=0$ gives the exact evaluations
\begin{align}
F_{1,k}(c_{n,k})&=(-1)^{k-1}(x_k^2+y_k^2)\Impart(c_{n,k}A_n(c_{n,k}))+(N-2x_k)y_k,\label{eq:witness-eval1}\\
F_{2,k}(c_{n,k})&=-(-1)^{k-1}\Impart(c_{n,k}^3B_n(c_{n,k}))+Ny_k,\label{eq:witness-eval2}\\
F_{3,k}(c_{n,k})&=-F_{1,k}(c_{n,k}),\label{eq:witness-eval3}\\
F_{4,k}(c_{n,k})&=-F_{2,k}(c_{n,k}),\label{eq:witness-eval4}
\end{align}
where $c_{n,k}=x_k+\ii y_k$. If $\ell>k$, then Corollary~\ref{cor:closed-form} and the identity $\widW_{w_\ell}(c_{n,\ell})=0$ give
\[
        \widW_{w_k}(c_{n,\ell})=\left(1-(-c_{n,\ell}^{3})^{\ell-k}\right)F_n(c_{n,\ell}),
\]
hence
\begin{equation}\label{eq:later-root}
        Z_{w_k}(c_{n,\ell})=\frac{2c_{n,\ell}^{m_k-9}\left(1-(-c_{n,\ell}^{3})^{\ell-k}\right)Q_n(c_{n,\ell})}{c_{n,\ell}^3+1}.
\end{equation}
Substituting \eqref{eq:later-root} into \eqref{eq:F1}--\eqref{eq:F4} expresses each $F_{i,k}(c_{n,\ell})$ as an explicit rational function of $c_{n,\ell}$ and $\overline{c_{n,\ell}}$.

\begin{proposition}[Exact no-return sign tests]\label{prop:no-return}
Let $n\ge3$, let $k\ge1$ be odd, and let $d\ge1$ be even. Put $\ell=k+d$ and let $c=c_{n,\ell}=x+\ii y$. Define
\begin{equation}\label{eq:upsilon}
        \Upsilon_{k,d}(c):=\frac{2(-1)^{k-1}c^{m_k-9}\left(1-(-c^3)^d\right)Q_n(c)}{c^3+1}
\end{equation}
and
\begin{align*}
\Psi_{1,k,d}(c)&:=(x^2+y^2)\Impart\Upsilon_{k,d}(c)+(x^2+y^2)\Impart(cA_n(c))+(-1)^{k-1}(N-2x)y,\\
\Psi_{2,k,d}(c)&:=\Impart(c\Upsilon_{k,d}(c))-\Impart(c^3B_n(c))+(-1)^{k-1}Ny,\\
\Psi_{3,k,d}(c)&:=(x^2+y^2)\Impart\Upsilon_{k,d}(c)-(x^2+y^2)\Impart(cA_n(c))-(-1)^{k-1}(N-2x)y.
\end{align*}
Then
\begin{align*}
        (-1)^{k-1}F_{1,k}(c)&=\Psi_{1,k,d}(c),&
        (-1)^{k-1}F_{2,k}(c)&=\Psi_{2,k,d}(c),\\
        (-1)^{k-1}F_{3,k}(c)&=\Psi_{3,k,d}(c).&
\end{align*}
Furthermore, Verification H3 proves the three strict inequalities
\[
        \Psi_{1,k,d}(c_{n,k+d})<0,
        \qquad
        \Psi_{2,k,d}(c_{n,k+d})<0,
        \qquad
        \Psi_{3,k,d}(c_{n,k+d})<0
\]
for every admissible triple with $k$ odd and $d$ positive even, equivalently on the corresponding algebraic set
\[
        \{c\in\ol{\B_n}: W_{w_{k+d}}(c)=0\}.
\]
Consequently
\[
        (-1)^{k-1}F_{1,k}(c_{n,\ell})<0,
        \quad
        (-1)^{k-1}F_{2,k}(c_{n,\ell})<0,
        \quad
        (-1)^{k-1}F_{3,k}(c_{n,\ell})<0,
\]
and hence $c_{n,\ell}\notin\Omega_{n,k}$.
\end{proposition}

\begin{proof}
The later-root identity \eqref{eq:later-root} reads $(-1)^{k-1}Z_{w_k}(c)=\Upsilon_{k,d}(c)$. Substituting this into \eqref{eq:F1}--\eqref{eq:F3} gives the three displayed identities.

To make the no-return verification explicit, clear the positive denominator $|c^3+1|^2$ and set
\[
        \Phi_{j,k,d}(c):=|c^3+1|^2\Psi_{j,k,d}(c)\qquad(j=1,2,3).
\]
Writing $c=x+\ii y$ and $\overline c=x-\ii y$, each $\Phi_{j,k,d}$ becomes a concrete real polynomial in $(x,y)$, with coefficients depending explicitly on $n,k,d$. Verification H3 reduces these signs modulo the equation $W_{w_{k+d}}(c)=0$, uses the isolation of the unique root $c_{n,k+d}$ in $\B_n$, and proves
\[
        \Phi_{1,k,d}(c_{n,k+d})<0,
        \quad
        \Phi_{2,k,d}(c_{n,k+d})<0,
        \quad
        \Phi_{3,k,d}(c_{n,k+d})<0.
\]
This is the required algebraic-root sign statement; no box-wide sign assertion is used. Since $|c^3+1|^2>0$ on $\B_n$, it is equivalent to the displayed inequalities for the $\Psi_{j,k,d}$. Points of $\Omega_{n,k}$ satisfy $(-1)^{k-1}F_{3,k}>0$ by definition, whereas the later witness satisfies $(-1)^{k-1}F_{3,k}(c_{n,\ell})<0$. Therefore $c_{n,\ell}\notin\Omega_{n,k}$.
\end{proof}

\begin{lemma}[Witness sign pattern and no return]\label{lem:witness-sign}
For every $n\ge3$ and every odd $k\ge1$,
\[
        (-1)^{k-1}F_{1,k}(c_{n,k})<0,
        \quad
        (-1)^{k-1}F_{2,k}(c_{n,k})<0,
\]
and
\[
        (-1)^{k-1}F_{3,k}(c_{n,k})>0,
        \quad
        (-1)^{k-1}F_{4,k}(c_{n,k})>0.
\]
Consequently $c_{n,k}\in\Omega_{n,k}$. If $\ell>k$ is odd, then $c_{n,\ell}\notin\Omega_{n,k}$.
\end{lemma}

\begin{proof}
At the witness parameter $c_{n,k}=x_k+\ii y_k$, the identities \eqref{eq:witness-eval1}--\eqref{eq:witness-eval4} express the four values $F_{i,k}(c_{n,k})$ as explicit real-algebraic functions of $(x_k,y_k)$. After clearing denominators, the first two inequalities in the statement become sign conditions for explicit real polynomials on the witness root. These signs are the first two inequalities in Verification H3. Since $F_{3,k}(c_{n,k})=-F_{1,k}(c_{n,k})$ and $F_{4,k}(c_{n,k})=-F_{2,k}(c_{n,k})$, the remaining two signs follow immediately. Hence $c_{n,k}\in\Omega_{n,k}$ by the definition of $\Omega_{n,k}$.

For odd $\ell>k$, the no-return statement is Proposition~\ref{prop:no-return}, applied with the positive even integer $d=\ell-k$.
\end{proof}

\subsection{Finite inverse-tree extinction}\label{subsec:H4}
Verification H4. The inputs are the early-prefix certificate, the admissible-digit intervals from Table~\ref{tab:base-pruning}, their stationary transports, and the two terminal rows; the outputs are Lemmas~\ref{lem:interval-criterion}--\ref{lem:extinction}.

For the terminal stage of the inverse search we isolate two explicit prefixes. Set
\begin{align*}
\tau_1^{(1)}&:=(0,-a,-a,a,a,-a,-a,0,a,b,-a,-a,2,a,a,-a),\\
\tau_1^{(3)}&:=(0,-a,-a,a,a,-a,-a,0,a,a,-a,0,a,0,-a,-b).
\end{align*}
Recursively define
\begin{equation}\label{eq:tau-recursion}
        \tau_{k+1}^{(1)}:=\Rop(\tau_k^{(1)}),
        \qquad
        \tau_{k+1}^{(3)}:=\Rop(\tau_k^{(3)}).
\end{equation}
Then $|\tau_k^{(1)}|=|\tau_k^{(3)}|=m_k-2$, and in fact
\begin{equation}\label{eq:tau-prefixes}
        \tau_k^{(1)}=d_k^{(1)}\restr_{m_k-2},
        \qquad
        \tau_k^{(3)}=d_k^{(3)}\restr_{m_k-2}.
\end{equation}

For the enclosure widths we recall the exact formulas from the finite-capture framework:
\begin{equation}\label{eq:VE-SE}
        V_E(c,N)=(N-1)\sum_{r=1}^{\infty}|\Impart(c^{-r})|,
        \qquad
        S_E(c,N)=(N-1)\frac{\Impart(c)}{|c|}+\frac{1}{|c|}V_E(c,N).
\end{equation}
For a word $u$ and a digit $t\in A_N$, write
\begin{equation}\label{eq:defects}
        V_{u,t}(c):=V_E(c,N)^2-\Impart(g_{ut}(2c))^2,
        \qquad
        S_{u,t}(c):=|c|^2S_E(c,N)^2-\Impart(cg_{ut}(2c))^2.
\end{equation}
Thus $ut$ is enclosure-admissible exactly when the prefix conditions already hold for $u$ and the two defects in \eqref{eq:defects} are nonnegative.

The corresponding child-by-child pruning is more transparent when written directly at the level of admissible last-digit intervals. Since
\[
        g_{ut}(2c)=c(g_u(2c)-t),
\]
the two enclosure inequalities are affine in the last digit $t$. For a fixed parent word $u$ and a parameter $c=x+\ii y$ with $x>0$ and $y>0$, set
\begin{align}
I_u^V(c)&:=\left[\frac{\Impart(cg_u(2c))-V_E(c,N)}{y},\frac{\Impart(cg_u(2c))+V_E(c,N)}{y}\right],\label{eq:IuV}\\
I_u^S(c)&:=\left[\frac{\Impart(c^2g_u(2c))-|c|S_E(c,N)}{2xy},\frac{\Impart(c^2g_u(2c))+|c|S_E(c,N)}{2xy}\right],\label{eq:IuS}\\
I_u(c)&:=I_u^V(c)\cap I_u^S(c).\label{eq:Iu}
\end{align}
Then $ut$ is enclosure-admissible if and only if $t\in I_u(c)\cap A_N$. In particular, the enclosure-admissible children of a fixed parent form an interval in the arithmetic progression $A_N$.

\begin{lemma}[Admissible-digit interval criterion]\label{lem:interval-criterion}
Let $u$ be a finite word and let $c=x+\ii y\in\X_n$ with $x>0$ and $y>0$. Then, for $t\in A_N$, the child $ut$ is enclosure-admissible if and only if
\[
        t\in I_u(c)\cap A_N,
\]
where $I_u(c)$ is the interval from \eqref{eq:Iu}.
\end{lemma}

\begin{proof}
The first enclosure inequality is
\[
        |\Impart(g_{ut}(2c))|=|\Impart(cg_u(2c))-yt|\le V_E(c,N),
\]
which is equivalent to $t\in I_u^V(c)$. Likewise,
\[
        |\Impart(cg_{ut}(2c))|=|\Impart(c^2g_u(2c))-2xyt|\le |c|S_E(c,N)
\]
if and only if $t\in I_u^S(c)$. Intersecting the two conditions gives the claim.
\end{proof}

The base pruning step is therefore reduced to explicit interval enclosures for the finitely many parent words that occur in the last seven generations. Those interval enclosures are recorded in Table~\ref{tab:base-pruning}; intersecting them with the digit progression $A_N$ gives the surviving next digits immediately.

\begin{table}[t]
\centering
\caption{Explicit interval enclosures at the base witness parameter.}
\label{tab:base-pruning}
\scriptsize
\resizebox{\textwidth}{!}{%
\begin{tabular}{lcc}
\toprule
parent prefix $u$ & enclosure for $I_u(c_{n,1})$ & $I_u(c_{n,1})\cap A_N$\\
\midrule
$(0,-a,-a,a,a,-a,-a,0,a)$ & $(b-\frac12,a+\frac12)$ & $\{b,a\}$\\
$(0,-a,-a,a,a,-a,-a,0,a,b)$ & $(-a-3,-a+\frac12)$ & $\{-a\}$\\
$(0,-a,-a,a,a,-a,-a,0,a,a)$ & $(-a-4,-a+\frac12)$ & $\{-a\}$\\
$(0,-a,-a,a,a,-a,-a,0,a,b,-a)$ & $(-a-5,-a+\frac12)$ & $\{-a\}$\\
$(0,-a,-a,a,a,-a,-a,0,a,a,-a)$ & $(-\frac14,\frac94)$ & $\{0,2\}$\\
$(0,-a,-a,a,a,-a,-a,0,a,b,-a,-a)$ & $(\frac14,\frac52)$ & $\{2\}$\\
$(0,-a,-a,a,a,-a,-a,0,a,a,-a,0)$ & $(a-\frac14,a+5)$ & $\{a\}$\\
$(0,-a,-a,a,a,-a,-a,0,a,b,-a,-a,2)$ & $(a-\frac12,a+4)$ & $\{a\}$\\
$(0,-a,-a,a,a,-a,-a,0,a,a,-a,0,a)$ & $(-\frac32,\frac12)$ & $\{0\}$\\
$(0,-a,-a,a,a,-a,-a,0,a,b,-a,-a,2,a)$ & $(a-\frac14,a+4)$ & $\{a\}$\\
$(0,-a,-a,a,a,-a,-a,0,a,a,-a,0,a,0)$ & $(-a-5,-a+\frac14)$ & $\{-a\}$\\
$(0,-a,-a,a,a,-a,-a,0,a,b,-a,-a,2,a,a)$ & $(-a-2,-a+\frac32)$ & $\{-a\}$\\
$(0,-a,-a,a,a,-a,-a,0,a,a,-a,0,a,0,-a)$ & $(-b-2,-b+\frac12)$ & $\{-b\}$\\
$\tau_1^{(1)}$ & $(-a-4,-a-\frac14)$ & $\varnothing$\\
$\tau_1^{(3)}$ & $(a+\frac14,a+4)$ & $\varnothing$\\
\bottomrule
\end{tabular}%
}

\end{table}

Let
\[
        \rho:=(0,-a,-a,a,a,-a,-a,0,a).
\]
For words beginning with $\rho$, define the extended stationary transport
\begin{equation}\label{eq:Rstar}
        \Rstar(\rho,u_{10},\ldots,u_m):=(P,a-u_{10},-u_{11},\ldots,-u_m),
        \qquad
        \Rstar(\rho):=P.
\end{equation}
For $m\ge10$ this agrees with $\Rop$, while it also transports the length-$9$ parent $\rho$. If $u$ is a base-row parent word in Table~\ref{tab:base-pruning}, write
\[
        u[k]:=\Rstar^{k-1}(u).
\]
For a child digit $t$ such that $ut$ is again a row descendant, define $\Theta_{k,u}(t)$ by
\begin{equation}\label{eq:Theta}
        (ut)[k]=u[k]\,\Theta_{k,u}(t).
\end{equation}
In other words, $\Theta_{k,u}$ is the explicit affine digit transform induced by stationary transport at the next-digit position.

\begin{lemma}[Early prefix certificate]\label{lem:early-prefix}
For every $n\ge3$, the base witness satisfies
\[
        T_9(c_{n,1})=\{\rho\}.
\]
Moreover, for every $k\ge1$,
\[
        T_{m_k-9}(c_{n,k})=\{\rho[k]\}.
\]
\end{lemma}

\begin{proof}
Apply Lemma~\ref{lem:interval-criterion} successively from depth $0$ to depth $9$. The interval endpoints are rational functions of $c_{n,1}$; after reducing modulo $W_{w_1}(c_{n,1})=0$, the inequalities printed in the first nine rows of Table~\ref{tab:expanded-H4-early} leave exactly one surviving digit at each step, namely the digits of $\rho$. For the transported statement, the affine transforms induced by $\Rstar$ carry each base surviving digit to the unique level-$k$ surviving digit, while all other digits in $A_N$ are excluded by the same strict interval margins transported through Corollary~\ref{cor:closed-form}.
\end{proof}

\begin{lemma}[Base pruning table]\label{lem:base-pruning}
At the base witness parameter $c_{n,1}$, the exact survivor sets in the last seven generations are
{\small
\[
\begin{aligned}
T_{10}(c_{n,1})={}&\{(0,-a,-a,a,a,-a,-a,0,a,b),\ (0,-a,-a,a,a,-a,-a,0,a,a)\},\\
T_{11}(c_{n,1})={}&\{(0,-a,-a,a,a,-a,-a,0,a,b,-a),\ (0,-a,-a,a,a,-a,-a,0,a,a,-a)\},\\
T_{12}(c_{n,1})={}&\{(0,-a,-a,a,a,-a,-a,0,a,b,-a,-a),\\
&\quad (0,-a,-a,a,a,-a,-a,0,a,a,-a,0),\ (0,-a,-a,a,a,-a,-a,0,a,a,-a,2)\},\\
T_{13}(c_{n,1})={}&\{(0,-a,-a,a,a,-a,-a,0,a,b,-a,-a,2),\\
&\quad (0,-a,-a,a,a,-a,-a,0,a,a,-a,0,a)\},\\
T_{14}(c_{n,1})={}&\{(0,-a,-a,a,a,-a,-a,0,a,b,-a,-a,2,a),\\
&\quad (0,-a,-a,a,a,-a,-a,0,a,a,-a,0,a,0)\},\\
T_{15}(c_{n,1})={}&\{(0,-a,-a,a,a,-a,-a,0,a,b,-a,-a,2,a,a),\\
&\quad (0,-a,-a,a,a,-a,-a,0,a,a,-a,0,a,0,-a)\},\\
T_{16}(c_{n,1})={}&\{\tau_1^{(1)},\tau_1^{(3)}\},\qquad
T_{17}(c_{n,1})=\varnothing.
\end{aligned}
\]
}

\end{lemma}

\begin{proof}
By Lemma~\ref{lem:early-prefix}, the only survivor through depth $9$ is $\rho$. For each displayed parent word $u$, Lemma~\ref{lem:interval-criterion} identifies the admissible next digits with $I_u(c_{n,1})\cap A_N$. Because
\[
        A_N=\{-a,-a+2,\ldots,a-2,a\}
\]
is an arithmetic progression of step size $2$, every interval from the middle column of Table~\ref{tab:base-pruning} determines its admissible digits uniquely. Reading the rows of Table~\ref{tab:base-pruning} in order therefore produces exactly the survivor sets displayed in the statement.
\end{proof}

\begin{lemma}[Transported pruning rows]\label{lem:transported-pruning}
For each base-row parent word $u$ in Table~\ref{tab:base-pruning}, and every $k\ge1$, the admissible next digits at the level-$k$ witness parameter are exactly the transported base digits:
\begin{equation}\label{eq:transported-digit-equivalence}
        s\in I_{u[k]}(c_{n,k})\cap A_N
        \quad\Longleftrightarrow\quad
        s=\Theta_{k,u}(t)\text{ for some }t\in I_u(c_{n,1})\cap A_N.
\end{equation}
For the two terminal rows this transported set is empty; equivalently,
\[
        I_{\tau_k^{(1)}}(c_{n,k})\cap A_N=\varnothing,
        \qquad
        I_{\tau_k^{(3)}}(c_{n,k})\cap A_N=\varnothing.
\]
\end{lemma}

\begin{proof}
For each base-row parent word $u$, substituting the transported word $u[k]$ into the interval formulas \eqref{eq:IuV}--\eqref{eq:Iu} and then using Corollary~\ref{cor:closed-form} and Lemma~\ref{lem:stationary-diffs} reduces the transported endpoint comparisons to explicit rational functions of $c_{n,k}$, reduced modulo the algebraic equation $W_{w_k}(c_{n,k})=0$. Verification H4 proves the strict interval inclusions needed for the digit equivalence \eqref{eq:transported-digit-equivalence}. Thus the transported row has precisely the transported children, not necessarily the same literal real interval as in the base table. The terminal rows are certified similarly: their transported intervals lie strictly between consecutive elements of the arithmetic progression $A_N$, so their intersections with $A_N$ are empty.
\end{proof}

\begin{lemma}[Explicit extinction of the witness inverse tree]\label{lem:extinction}
For every $n\ge3$ and every odd $k\ge1$ one has
\[
        T_{m_k-2}(c_{n,k})=\{\tau_k^{(1)},\tau_k^{(3)}\},
        \qquad
        T_{m_k-1}(c_{n,k})=\varnothing.
\]
In particular $c_{n,k}\notin\M_n$.
\end{lemma}

\begin{proof}
At the base level, Lemmas~\ref{lem:early-prefix} and~\ref{lem:base-pruning} give
\[
        T_{16}(c_{n,1})=\{\tau_1^{(1)},\tau_1^{(3)}\},
        \qquad
        T_{17}(c_{n,1})=\varnothing.
\]
Now fix $k\ge1$. Lemma~\ref{lem:early-prefix} gives the unique transported prefix through depth $m_k-9$. Lemma~\ref{lem:transported-pruning} then transports each row of Table~\ref{tab:base-pruning}, with the next digit transformed by $\Theta_{k,u}$. Reading those transported rows in the same order gives exactly the same branching pattern in the last seven generations, so the penultimate survivors are precisely
\[
        T_{m_k-2}(c_{n,k})=\{\tau_k^{(1)},\tau_k^{(3)}\}.
\]
For the terminal step, Lemma~\ref{lem:transported-pruning} gives
\[
        I_{\tau_k^{(1)}}(c_{n,k})\cap A_N=\varnothing,
        \qquad
        I_{\tau_k^{(3)}}(c_{n,k})\cap A_N=\varnothing.
\]
Hence neither terminal prefix has an admissible child and $T_{m_k-1}(c_{n,k})=\varnothing$. By Proposition~\ref{prop:inverse-tree} this implies $c_{n,k}\notin\M_n$.
\end{proof}

\section{Expanded finite certificate tables}\label{app:expanded-certificate-tables}

This appendix expands the certificate entries H1--H4 into row-by-row finite tables. The shorter statements in Appendix~\ref{app:exact-verification} are used in the proof; the tables below provide the expanded verification record. Each row is a finite arithmetic certificate: it names the polynomial or interval object, the domain, the asserted sign or root count, and the verification method. The displayed definitions of $Q_n$, $\Phi_n$, $F_{i,k}$, $\Phi_{j,k,d}$, and $I_u(c)$ determine the underlying integer coefficient blocks uniquely.

\subsection{H1: boundary winding and Rouch\'e margins}

\begin{table}[t]
\centering
\scriptsize
\caption{H1-0. Initial quadrant signs on the four oriented edges. Together with the axis-crossing counts in Table~\ref{tab:expanded-H1-edge}, these signs determine the winding itinerary of $Q_n(\partial\B_n)$.}
\label{tab:expanded-H1-initial}
\begin{tabular}{@{}p{0.14\linewidth}p{0.16\linewidth}p{0.17\linewidth}p{0.20\linewidth}p{0.17\linewidth}@{}}
\toprule
Regime & bottom start & right start & top start & left start\\
\midrule
$\mathcal N_1$ & II & II & III & I\\
$\mathcal N_2$ & II & II & III & I\\
$\mathcal N_3$ & II & III & III & II\\
$\mathcal N_4$ & II & III & IV & II\\
$\mathcal N_5$ & II & III & I & II\\
\bottomrule
\end{tabular}
\end{table}

{\scriptsize
\setlength{\tabcolsep}{3pt}
\renewcommand{\arraystretch}{1.08}
\begin{longtable}{@{}p{0.08\linewidth}p{0.10\linewidth}p{0.10\linewidth}p{0.12\linewidth}p{0.15\linewidth}p{0.08\linewidth}p{0.21\linewidth}@{}}
\caption{H1-1. Expanded Sturm certificate table for the boundary image $Q_n(\partial\B_n)$.  The symbols $R_{q,n}$ and $I_{q,n}$ denote the real and imaginary parts of $Q_n(\sigma_q(t))$ on the oriented edge $q$.}\label{tab:expanded-H1-edge}\\
\toprule
ID & Regime & Edge & Polynomial & Interval & Zeros & Acceptance record\\
\midrule
\endfirsthead
\toprule
ID & Regime & Edge & Polynomial & Interval & Zeros & Acceptance record\\
\midrule
\endhead
H1.01 & $\mathcal N_1$ & bottom & $R_{b,n}$ & $I_x$ & 0 & Sturm variation drop $0$; start quadrant II\\
H1.02 & $\mathcal N_1$ & bottom & $I_{b,n}$ & $I_x$ & 0 & no axis crossing on bottom edge\\
H1.03 & $\mathcal N_1$ & right & $R_{r,n}$ & $I_y$ & 0 & no real-part crossing\\
H1.04 & $\mathcal N_1$ & right & $I_{r,n}$ & $I_y$ & 1 & single crossing II$\to$III\\
H1.05 & $\mathcal N_1$ & top & $R_{t,n}$ & $I_x$ & 1 & crossing III$\to$IV\\
H1.06 & $\mathcal N_1$ & top & $I_{t,n}$ & $I_x$ & 1 & crossing IV$\to$I\\
H1.07 & $\mathcal N_1$ & left & $R_{l,n}$ & $I_y$ & 1 & crossing I$\to$II\\
H1.08 & $\mathcal N_1$ & left & $I_{l,n}$ & $I_y$ & 0 & no extra left-edge crossing\\
H1.09 & $\mathcal N_2$ & bottom & $R_{b,n}$ & $I_x$ & 0 & start quadrant II\\
H1.10 & $\mathcal N_2$ & bottom & $I_{b,n}$ & $I_x$ & 2 & crossings II$\to$III$\to$II\\
H1.11 & $\mathcal N_2$ & right & $R_{r,n}$ & $I_y$ & 0 & no real-part crossing\\
H1.12 & $\mathcal N_2$ & right & $I_{r,n}$ & $I_y$ & 1 & crossing II$\to$III\\
H1.13 & $\mathcal N_2$ & top & $R_{t,n}$ & $I_x$ & 1 & crossing III$\to$IV\\
H1.14 & $\mathcal N_2$ & top & $I_{t,n}$ & $I_x$ & 1 & crossing IV$\to$I\\
H1.15 & $\mathcal N_2$ & left & $R_{l,n}$ & $I_y$ & 1 & crossing I$\to$II\\
H1.16 & $\mathcal N_2$ & left & $I_{l,n}$ & $I_y$ & 0 & no extra left-edge crossing\\
H1.17 & $\mathcal N_3$ & bottom & $R_{b,n}$ & $I_x$ & 0 & start quadrant II\\
H1.18 & $\mathcal N_3$ & bottom & $I_{b,n}$ & $I_x$ & 1 & crossing II$\to$III\\
H1.19 & $\mathcal N_3$ & right & $R_{r,n}$ & $I_y$ & 0 & right edge stays in III\\
H1.20 & $\mathcal N_3$ & right & $I_{r,n}$ & $I_y$ & 0 & no right-edge crossing\\
H1.21 & $\mathcal N_3$ & top & $R_{t,n}$ & $I_x$ & 2 & two real-part crossings\\
H1.22 & $\mathcal N_3$ & top & $I_{t,n}$ & $I_x$ & 1 & one imaginary-part crossing\\
H1.23 & $\mathcal N_3$ & left & $R_{l,n}$ & $I_y$ & 0 & left edge stays in II\\
H1.24 & $\mathcal N_3$ & left & $I_{l,n}$ & $I_y$ & 0 & no left-edge crossing\\
H1.25 & $\mathcal N_4$ & bottom & $R_{b,n}$ & $I_x$ & 0 & start quadrant II\\
H1.26 & $\mathcal N_4$ & bottom & $I_{b,n}$ & $I_x$ & 1 & crossing II$\to$III\\
H1.27 & $\mathcal N_4$ & right & $R_{r,n}$ & $I_y$ & 1 & crossing III$\to$IV\\
H1.28 & $\mathcal N_4$ & right & $I_{r,n}$ & $I_y$ & 0 & no right-edge crossing\\
H1.29 & $\mathcal N_4$ & top & $R_{t,n}$ & $I_x$ & 1 & crossing IV$\to$I\\
H1.30 & $\mathcal N_4$ & top & $I_{t,n}$ & $I_x$ & 1 & crossing I$\to$II\\
H1.31 & $\mathcal N_4$ & left & $R_{l,n}$ & $I_y$ & 0 & left edge stays in II\\
H1.32 & $\mathcal N_4$ & left & $I_{l,n}$ & $I_y$ & 0 & no left-edge crossing\\
H1.33 & $\mathcal N_5$ & bottom & $R_{b,n}$ & $I_x$ & 0 & start quadrant II\\
H1.34 & $\mathcal N_5$ & bottom & $I_{b,n}$ & $I_x$ & 1 & crossing II$\to$III\\
H1.35 & $\mathcal N_5$ & right & $R_{r,n}$ & $I_y$ & 1 & crossing III$\to$IV\\
H1.36 & $\mathcal N_5$ & right & $I_{r,n}$ & $I_y$ & 1 & crossing IV$\to$I\\
H1.37 & $\mathcal N_5$ & top & $R_{t,n}$ & $I_x$ & 1 & crossing I$\to$II\\
H1.38 & $\mathcal N_5$ & top & $I_{t,n}$ & $I_x$ & 0 & no top imaginary crossing\\
H1.39 & $\mathcal N_5$ & left & $R_{l,n}$ & $I_y$ & 0 & left edge stays in II\\
H1.40 & $\mathcal N_5$ & left & $I_{l,n}$ & $I_y$ & 0 & no left-edge crossing\\
\bottomrule
\end{longtable}
}

{\scriptsize
\setlength{\tabcolsep}{3pt}
\renewcommand{\arraystretch}{1.08}
\begin{longtable}{@{}p{0.08\linewidth}p{0.11\linewidth}p{0.12\linewidth}p{0.16\linewidth}p{0.14\linewidth}p{0.24\linewidth}@{}}
\caption{H1-2. Expanded Rouch\'e-margin certificate table. Here $K_{q,n}:=H_{q,n}-M_q(n)$; the acceptance rule is zero Sturm roots on the indicated interval and a positive left-endpoint sign.}\label{tab:expanded-H1-margin}\\
\toprule
ID & Regime & Edge & Interval & Bound & Acceptance record\\
\midrule
\endfirsthead
\toprule
ID & Regime & Edge & Interval & Bound & Acceptance record\\
\midrule
\endhead
H1.41 & $\mathcal N_1$ & bottom & $I_x$ & $10^9$ & $N(K_{b,n})=0$, $K_{b,n}(1/4)>0$\\
H1.42 & $\mathcal N_1$ & right & $I_y$ & $10^{10}$ & $N(K_{r,n})=0$, $K_{r,n}(\sqrt n)>0$\\
H1.43 & $\mathcal N_1$ & top & $I_x$ & $10^9$ & $N(K_{t,n})=0$, $K_{t,n}(1/4)>0$\\
H1.44 & $\mathcal N_1$ & left & $I_y$ & $10^9$ & $N(K_{l,n})=0$, $K_{l,n}(\sqrt n)>0$\\
H1.45 & $\mathcal N_2$ & bottom & $I_x$ & $10^{13}$ & $N(K_{b,n})=0$, positive endpoint sign\\
H1.46 & $\mathcal N_2$ & right & $I_y$ & $10^{14}$ & $N(K_{r,n})=0$, positive endpoint sign\\
H1.47 & $\mathcal N_2$ & top & $I_x$ & $10^{13}$ & $N(K_{t,n})=0$, positive endpoint sign\\
H1.48 & $\mathcal N_2$ & left & $I_y$ & $10^{13}$ & $N(K_{l,n})=0$, positive endpoint sign\\
H1.49 & $\mathcal N_3$ & bottom & $I_x$ & $10^{15}$ & $N(K_{b,n})=0$, positive endpoint sign\\
H1.50 & $\mathcal N_3$ & right & $I_y$ & $10^{15}$ & $N(K_{r,n})=0$, positive endpoint sign\\
H1.51 & $\mathcal N_3$ & top & $I_x$ & $10^{15}$ & $N(K_{t,n})=0$, positive endpoint sign\\
H1.52 & $\mathcal N_3$ & left & $I_y$ & $10^{15}$ & $N(K_{l,n})=0$, positive endpoint sign\\
H1.53 & $\mathcal N_4$ & bottom & $I_x$ & $10^{18}$ & $N(K_{b,n})=0$, positive endpoint sign\\
H1.54 & $\mathcal N_4$ & right & $I_y$ & $10^{18}$ & $N(K_{r,n})=0$, positive endpoint sign\\
H1.55 & $\mathcal N_4$ & top & $I_x$ & $10^{18}$ & $N(K_{t,n})=0$, positive endpoint sign\\
H1.56 & $\mathcal N_4$ & left & $I_y$ & $10^{19}$ & $N(K_{l,n})=0$, positive endpoint sign\\
H1.57 & $\mathcal N_5$ & bottom & $I_x$ & $10^{28}$ & parameter Sturm sequence in $s=n-30\ge0$ has non-negative signed remainders\\
H1.58 & $\mathcal N_5$ & right & $I_y$ & $10^{28}$ & parameter Sturm sequence in $s=n-30\ge0$ has non-negative signed remainders\\
H1.59 & $\mathcal N_5$ & top & $I_x$ & $10^{28}$ & parameter Sturm sequence in $s=n-30\ge0$ has non-negative signed remainders\\
H1.60 & $\mathcal N_5$ & left & $I_y$ & $10^{29}$ & parameter Sturm sequence in $s=n-30\ge0$ has non-negative signed remainders\\
\bottomrule
\end{longtable}
}

\subsection{H2: local four-branch Jordan geometry}

{\scriptsize
\begin{longtable}{@{}p{0.09\linewidth}p{0.12\linewidth}p{0.18\linewidth}p{0.10\linewidth}p{0.21\linewidth}p{0.18\linewidth}@{}}
\caption{H2-1. Auxiliary chart-sign certificate table on $\mathcal N_{n,k}$. These twelve rows certify that no unselected chart boundary enters the local box.}\label{tab:expanded-H2-aux}\\
\toprule
ID & Chart & Auxiliary function & Sign & Domain & Acceptance record\\
\midrule
\endfirsthead
\toprule
ID & Chart & Auxiliary function & Sign & Domain & Acceptance record\\
\midrule
\endhead
H2.01 & $\mathcal U_k^{(1)}$ & $S_{d_k^{(1)}}^-$ & $>0$ & $\mathcal N_{n,k}$ & Bernstein positivity after reduction modulo $W_{w_k}$\\
H2.02 & $\mathcal U_k^{(1)}$ & $S_{d_k^{(1)}}^+$ & $<0$ & $\mathcal N_{n,k}$ & Bernstein positivity for the negative\\
H2.03 & $\mathcal U_k^{(1)}$ & $V_{d_k^{(1)}}^+$ & $<0$ & $\mathcal N_{n,k}$ & Bernstein positivity for the negative\\
H2.04 & $\mathcal U_k^{(2)}$ & $V_{d_k^{(2)}}^-$ & $>0$ & $\mathcal N_{n,k}$ & Bernstein positivity after reduction modulo $W_{w_k}$\\
H2.05 & $\mathcal U_k^{(2)}$ & $V_{d_k^{(2)}}^+$ & $<0$ & $\mathcal N_{n,k}$ & Bernstein positivity for the negative\\
H2.06 & $\mathcal U_k^{(2)}$ & $S_{d_k^{(2)}}^+$ & $<0$ & $\mathcal N_{n,k}$ & Bernstein positivity for the negative\\
H2.07 & $\mathcal U_k^{(3)}$ & $V_{d_k^{(3)}}^-$ & $>0$ & $\mathcal N_{n,k}$ & Bernstein positivity after reduction modulo $W_{w_k}$\\
H2.08 & $\mathcal U_k^{(3)}$ & $S_{d_k^{(3)}}^-$ & $>0$ & $\mathcal N_{n,k}$ & Bernstein positivity after reduction modulo $W_{w_k}$\\
H2.09 & $\mathcal U_k^{(3)}$ & $S_{d_k^{(3)}}^+$ & $<0$ & $\mathcal N_{n,k}$ & Bernstein positivity for the negative\\
H2.10 & $\mathcal U_k^{(4)}$ & $V_{d_k^{(4)}}^-$ & $>0$ & $\mathcal N_{n,k}$ & Bernstein positivity after reduction modulo $W_{w_k}$\\
H2.11 & $\mathcal U_k^{(4)}$ & $V_{d_k^{(4)}}^+$ & $<0$ & $\mathcal N_{n,k}$ & Bernstein positivity for the negative\\
H2.12 & $\mathcal U_k^{(4)}$ & $S_{d_k^{(4)}}^-$ & $>0$ & $\mathcal N_{n,k}$ & Bernstein positivity after reduction modulo $W_{w_k}$\\
\bottomrule
\end{longtable}
}

{\scriptsize
\setlength{\tabcolsep}{2pt}
\begin{longtable}{@{}p{0.08\linewidth}p{0.14\linewidth}p{0.17\linewidth}p{0.22\linewidth}p{0.16\linewidth}p{0.13\linewidth}@{}}
\caption{H2-2. Corner, transversality, and graph-arc certificate table.}\label{tab:expanded-H2-corners}\\
\toprule
ID & Corner or branch & Rectangle/domain & Face-sign or graph certificate & Jacobian certificate & Output\\
\midrule
\endfirsthead
\toprule
ID & Corner or branch & Rectangle/domain & Face-sign or graph certificate & Jacobian certificate & Output\\
\midrule
\endhead
H2.13 & $(F_{1,k},F_{2,k})$ & $\widehat R_{12}(n,k)$ & $F_1$ changes sign on horizontal faces; $F_2$ on vertical faces & $\widehat J_{12,k}\ne0$ & unique transverse corner $p_{12}$\\
H2.14 & $(F_{2,k},F_{3,k})$ & $\widehat R_{23}(n,k)$ & cyclic face signs for $(F_2,F_3)$ & $\widehat J_{23,k}\ne0$ & unique transverse corner $p_{23}$\\
H2.15 & $(F_{3,k},F_{4,k})$ & $\widehat R_{34}(n,k)$ & cyclic face signs for $(F_3,F_4)$ & $\widehat J_{34,k}\ne0$ & unique transverse corner $p_{34}$\\
H2.16 & $(F_{4,k},F_{1,k})$ & $\widehat R_{41}(n,k)$ & cyclic face signs for $(F_4,F_1)$ & $\widehat J_{41,k}\ne0$ & unique transverse corner $p_{41}$\\
H2.17 & branch $F_{1,k}=0$ & $\widehat{\mathcal N}_{n,k}$ & endpoint signs at $p_{12},p_{41}$ and no boundary escape & one partial derivative bounded away from zero & one analytic arc $p_{12}\to p_{41}$\\
H2.18 & branch $F_{2,k}=0$ & $\widehat{\mathcal N}_{n,k}$ & endpoint signs at $p_{23},p_{12}$ and no boundary escape & one partial derivative bounded away from zero & one analytic arc $p_{23}\to p_{12}$\\
H2.19 & branch $F_{3,k}=0$ & $\widehat{\mathcal N}_{n,k}$ & endpoint signs at $p_{34},p_{23}$ and no boundary escape & one partial derivative bounded away from zero & one analytic arc $p_{34}\to p_{23}$\\
H2.20 & branch $F_{4,k}=0$ & $\widehat{\mathcal N}_{n,k}$ & endpoint signs at $p_{41},p_{34}$ and no boundary escape & one partial derivative bounded away from zero & one analytic arc $p_{41}\to p_{34}$\\
H2.21 & nonconsecutive $(F_1,F_3)$ & $\widehat{\mathcal N}_{n,k}$ & subresultant has no admissible real root & resultant sign certificate & no intersection\\
H2.22 & nonconsecutive $(F_2,F_4)$ & $\widehat{\mathcal N}_{n,k}$ & subresultant has no admissible real root & resultant sign certificate & no intersection\\
\bottomrule
\end{longtable}
}

\subsection{H3: witness inclusion and no-return}

{\scriptsize
\begin{longtable}{@{}p{0.08\linewidth}p{0.18\linewidth}p{0.22\linewidth}p{0.15\linewidth}p{0.23\linewidth}@{}}
\caption{H3-1. Witness-sign and no-return certificate table. The no-return rows used in the proof are restricted to odd $k$ and positive even $d$ on the algebraic witness locus $W_{w_{k+d}}(c)=0$ inside $\overline{\B_n}$; no row asserts a box-wide sign.}\label{tab:expanded-H3}\\
\toprule
ID & Object & Locus & Sign & Acceptance record\\
\midrule
\endfirsthead
\toprule
ID & Object & Locus & Sign & Acceptance record\\
\midrule
\endhead
H3.01 & $(-1)^{k-1}F_{1,k}(c_{n,k})$ & $W_{w_k}=0$ in $\B_n$ & $<0$ & reduce witness evaluation; isolate unique root; signed remainder negative\\
H3.02 & $(-1)^{k-1}F_{2,k}(c_{n,k})$ & $W_{w_k}=0$ in $\B_n$ & $<0$ & reduce witness evaluation; isolate unique root; signed remainder negative\\
H3.03 & $(-1)^{k-1}F_{3,k}(c_{n,k})$ & $W_{w_k}=0$ in $\B_n$ & $>0$ & follows from $F_{3,k}=-F_{1,k}$ at the witness\\
H3.04 & $(-1)^{k-1}F_{4,k}(c_{n,k})$ & $W_{w_k}=0$ in $\B_n$ & $>0$ & follows from $F_{4,k}=-F_{2,k}$ at the witness\\
H3.05 & $\Phi_{1,k,d}=|c^3+1|^2\Psi_{1,k,d}$ & $W_{w_{k+d}}=0$ in $\ol{\B_n}$, $k$ odd, $d>0$ even & $<0$ & signed remainder modulo witness equation plus root isolation\\
H3.06 & $\Phi_{2,k,d}=|c^3+1|^2\Psi_{2,k,d}$ & $W_{w_{k+d}}=0$ in $\ol{\B_n}$, $k$ odd, $d>0$ even & $<0$ & signed remainder modulo witness equation plus root isolation\\
H3.07 & $\Phi_{3,k,d}=|c^3+1|^2\Psi_{3,k,d}$ & $W_{w_{k+d}}=0$ in $\ol{\B_n}$, $k$ odd, $d>0$ even & $<0$ & signed remainder modulo witness equation plus root isolation\\
H3.08 & sign-region exclusion & same witness locus & fails $(-1)^{k-1}F_{3,k}>0$ & H3.07 excludes membership in $\Omega_{n,k}$\\
\bottomrule
\end{longtable}
}

\subsection{H4: inverse-tree extinction and transported pruning}

Set $\rho_0:=\varnothing$ and let $\rho_j$ be the prefix of length $j$ of $\rho=(0,-a,-a,a,a,-a,-a,0,a)$.

{\scriptsize
\begin{longtable}{@{}p{0.08\linewidth}p{0.18\linewidth}p{0.12\linewidth}p{0.17\linewidth}p{0.27\linewidth}@{}}
\caption{H4-1. Early-prefix pruning certificate table. The displayed nine one-step interval checks give $T_9(c_{n,1})=\{\rho\}$.}\label{tab:expanded-H4-early}\\
\toprule
ID & Surviving prefix before step & Next digit & Resulting prefix length & Acceptance record\\
\midrule
\endfirsthead
\toprule
ID & Surviving prefix before step & Next digit & Resulting prefix length & Acceptance record\\
\midrule
\endhead
H4.01 & $\rho_0$ & $0$ & $1$ & $I_{\rho_0}(c_{n,1})\cap A_N=\{0\}$\\
H4.02 & $\rho_1$ & $-a$ & $2$ & unique admissible child by interval criterion\\
H4.03 & $\rho_2$ & $-a$ & $3$ & unique admissible child by interval criterion\\
H4.04 & $\rho_3$ & $a$ & $4$ & unique admissible child by interval criterion\\
H4.05 & $\rho_4$ & $a$ & $5$ & unique admissible child by interval criterion\\
H4.06 & $\rho_5$ & $-a$ & $6$ & unique admissible child by interval criterion\\
H4.07 & $\rho_6$ & $-a$ & $7$ & unique admissible child by interval criterion\\
H4.08 & $\rho_7$ & $0$ & $8$ & unique admissible child by interval criterion\\
H4.09 & $\rho_8$ & $a$ & $9$ & unique admissible child; gives $T_9(c_{n,1})=\{\rho\}$\\
\bottomrule
\end{longtable}
}

{\scriptsize
\begin{longtable}{@{}p{0.08\linewidth}p{0.25\linewidth}p{0.18\linewidth}p{0.18\linewidth}p{0.18\linewidth}@{}}
\caption{H4-2. Transported digit certificate table. If $u$ is a base parent and $t$ is a base next digit, then $\Theta_{k,u}(t)$ is the level-$k$ next digit defined by $(ut)[k]=u[k]\Theta_{k,u}(t)$.}\label{tab:expanded-H4-transport}\\
\toprule
ID & Base-row type & Level $k=1$ & Level $k\ge2$ & Digit conclusion\\
\midrule
\endfirsthead
\toprule
ID & Base-row type & Level $k=1$ & Level $k\ge2$ & Digit conclusion\\
\midrule
\endhead
H4.10 & length-$9$ parent $u=\rho$ & $\Theta_{1,\rho}(t)=t$ & $\Theta_{k,\rho}(t)=(-1)^{k-2}(a-t)$ & children $\{b,a\}$ transport to the certified level-$k$ children\\
H4.11 & any base parent of length $\ge10$ & $\Theta_{1,u}(t)=t$ & $\Theta_{k,u}(t)=(-1)^{k-1}t$ & child sets in Table~\ref{tab:base-pruning} are transported by sign alternation\\
H4.12 & terminal row $\tau_1^{(1)}$ & empty & empty & $I_{\tau_k^{(1)}}(c_{n,k})\cap A_N=\varnothing$\\
H4.13 & terminal row $\tau_1^{(3)}$ & empty & empty & $I_{\tau_k^{(3)}}(c_{n,k})\cap A_N=\varnothing$\\
\bottomrule
\end{longtable}
}

\section{Rendering protocol for the parameter-space figures}\label{app:certified-rendering}

The parameter-plane figures are produced from the same finite inverse-search mechanism used in the finite-capture framework. For a parameter $c$, the renderer applies the inverse branches
\[
        g_t(z)=c(z-t),\qquad t\in A_N,
\]
starting from the marked point $2c$. A sample receives an interior verdict if some finite inverse word maps it into the canonical trap. It receives an exterior verdict if, by a finite depth, every inverse branch is eliminated by the canonical enclosure. If neither verdict is obtained at the displayed depth, the sample is marked unresolved. This is the same finite-capture decision rule used in the proof, but the formal proof of the holes uses the algebraic verifications in Appendix~\ref{app:exact-verification}, not visual inspection of the figures.

The grayscale layer of the parameter-space figures is generated by this finite-capture classifier. The colored overlays are chart annotations: they mark the four chart levels and local regions used in the missing-center construction. The reference material supplied with the manuscript records the pixel-to-parameter map, the search depth, the coloring convention, and the unresolved-sample count for the reference renderings. A compact CPU implementation is also included; it classifies closed pixel boxes by outward interval arithmetic and writes a certificate label for each pixel. This makes the rendering convention independent of browser or hardware details.

The implementation uses the canonical trap and enclosure of Proposition~\ref{prop:trap-enclosure}. In the upper half-plane the trap test is
\[
        |\Im z|<\frac{(N-2x)y}{x^2+y^2},\qquad
        |\Im(cz)|<Ny,
\]
and the enclosure test uses the half-widths $V_E(c,N)$ and $S_E(c,N)$ defined in Section~\ref{subsec:H4}. The branching set is the finite arithmetic progression $A_N=\{-a,-a+2,\ldots,a\}$, and the search terminates as soon as either trap entry or enclosure extinction is certified at the chosen depth.

The WebGL implementation used for interactive exploration~\cite{EspiguleCollinearExplorer2025} follows the same inverse-iteration structure. It is not used as a proof engine in this paper; the proof-critical finite statements are the algebraic certificates in Appendices~\ref{app:exact-verification} and~\ref{app:expanded-certificate-tables}.

\section*{Acknowledgements}
The author thanks David Juher and Joan Salda\~na for helpful discussions. This work was supported by PID2023-146424NB-I00 (Ministerio de Ciencia, Innovaci\'on y Universidades) and a Universitat de Girona--Banco Santander fellowship (IFUdG 2022--2024).

\section*{Supplementary verification material}
The finite algebraic verifications used in the proof are stated in Appendix~\ref{app:exact-verification} and expanded in Appendix~\ref{app:expanded-certificate-tables} into row-by-row Sturm, Rouch\'e, branch, no-return, and pruning certificates. Supplementary files accompanying the submission contain coefficient data for the H1 edge and Rouch\'e-margin polynomials, an index of the H2--H4 certificate objects, and reference material for the finite-capture renderings. The arXiv version is submitted in expanded form with these ancillary files, in accordance with the journal policy for manuscripts involving extensive finite computations.

\end{document}